\documentclass[11pt]{article}

\usepackage{fullpage}
\usepackage{amsthm}
\usepackage{amssymb}
\usepackage{amsmath}
\usepackage{graphicx}
\usepackage{tikz}
\usepackage{mathrsfs}
\usepackage{float}
\usepackage{makecell}
\usepackage{amscd}
\usepackage{multirow}
\usepackage{hhline}

\newcommand{\Z}{\mathbb{Z}}
\newcommand{\N}{\mathbb{N}}

\newcommand{\Q}{\mathbb{Q}}

\renewcommand{\Im}{\text{Im}}

\DeclareMathOperator{\skel}{skel}

\usepackage{relsize}
\newtheorem{theorem}{Theorem}
\newtheorem{lemma}{Lemma}
\newtheorem*{unnumtheorem}{Theorem}
\newtheorem*{unnumclaim}{Claim}
\newtheorem*{unnumlemma}{Lemma}
\newtheorem*{lovaszlemma}{Lov\'asz Local Lemma}

\newtheorem{corollary}[theorem]{Corollary}

\theoremstyle{definition}

\newtheorem*{definition}{Definition}
\newtheorem{question}{Question}
\title{Small simplicial complexes with prescribed torsion in homology}
\author{Andrew Newman \thanks{The Ohio State University; partially supported by the National Science Foundation grant NSF-DMS \#1547357}}
\date{\today}

\begin{document}
\maketitle
\abstract
For $d \geq 2$ and $G$ a finite abelian group, define $T_d(G)$ to be the minimum number of vertices $n$ so that there exists a simplicial complex $X$ on $n$ vertices which has the torsion part of $H_{d - 1}(X)$ isomorphic to $G$. Here we use the probabilistic method, in particular the Lov\'asz Local Lemma, to establish an upper bound on $T_d(G)$ which matches the known lower bound up to a constant factor. That is, we prove that for every $d \geq 2$ there exist constants $c_d$ and $C_d$ so that for any finite abelian group $$c_d(\log |G|)^{1/d} \leq T_d(G) \leq C_d(\log |G|)^{1/d}.$$

\section{Introduction}
For an abelian group $A$, let $A_T$ denote the torsion subgroup of $A$. Given a $d$-dimensional simplicial complex $X$ on $n$ vertices, it is natural to ask which finite abelian groups could appear as $H_{d - 1}(X)_T$. One of the earliest, and perhaps most surprising, answers to this question comes from Kalai's groundbreaking paper on $\Q$-acyclic complexes \cite{Kalai}. A $d$-dimensional $\Q$-acyclic complex on $n$ vertices is defined in \cite{Kalai} to be a simplicial complex $X$ on $n$ vertices with complete $(d-1)$-skeleton so that $H_d(X) = 0$ and $H_{d - 1}(X)$ is finite. Thus $\Q$-acyclic complexes are higher-dimensional analogues of trees, and the main result of \cite{Kalai} is the following generalization of Cayley's formula: 
\begin{unnumtheorem}[Theorem 1 of \cite{Kalai}]
For any $n \in \N$ and any dimension $d \geq 2$, let $\mathcal{C}_{n, d}$ denote the collection of $d$-dimensional $\Q$-acyclic complexes on vertex set $[n]$ then 
$$\sum_{X \in \mathcal{C}_{n, d}} |H_{d - 1}(X)|^2 = n^{\binom{n - 2}{d}}.$$
\end{unnumtheorem}

As a corollary to this result, Kalai shows that for every dimension $d \geq 2$ there is a positive constant $k_d$ so that, 
$$\mathbb{E}(|H_{d - 1}(X)|^2) \geq \exp(2k_d n^d)$$
where $X$ is taken uniformly from $\mathcal{C}_{n, d}$. Furthermore, \cite{Kalai} observes that the maximum size of $H_{d - 1}(X)_T$ for $X$ a $d$-dimensional simplicial complex on $n$ vertices is bounded above by $\exp(K_d n^d)$ for some constant $K_d$ depending only on $d$. Thus for every $n \in \N$ and $d \geq 2$,
$$\exp(k_d n^d) \leq \max_{X \subseteq \Delta^{n - 1}} |H_{d - 1}(X)_T| \leq \exp(K_dn^d).$$

This result establishes the existence of small complexes with exceptionally large torsion in homology. Furthermore, we can observe the phenomenon of enormous torsion in homology explicitly, though currently only for a few classes of simplicial complexes. For example, \cite{LMR} and \cite{Meshulam} provide one of the only known classes of explicit constructions for $\Q$-acyclic complexes. These complexes are called sum complexes, and may be observed to have large torsion in homology. For $n \in \N$ and $A$ a subset of $\Z/n\Z$, \cite{LMR} defines the sum complex $X_A$ on $n$ vertices to be the $(|A| - 1)$-dimensional complex with vertex set $\Z/n\Z$, complete $(|A| - 2)$-skeleton and each possible top-dimensional face included if and only if the sum of the vertices that determine it belong to $A$. The main result of \cite{LMR} is that if $n$ is prime then $X_A$ is always a $\Q$-acyclic complex. We seem to get interesting examples of simplicial complexes with torsion in homology from this class of simplicial complexes. For example if $X$ is the sum complex  $X_{\{0, 1, 3\}}$ on 41 vertices then
$$H_1(X) \cong \Z/83\Z \oplus \Z/83\Z \oplus \Z/313{,}156{,}754{,}870{,}106{,}981{,}917{,}996{,}329{,}463 \Z.$$
Table \ref{tbl:sumcomplexes} gives some examples of $\Q$-acyclic sum complexes to show how large torsion can be for a $d$-complex on $n$ vertices. The dimension of each example is implicit from the set $A$, since $d = |A| - 1$.\\

\begin{table}[h]
  \centering
  \renewcommand{\arraystretch}{1.2}
  \begin{tabular}{|c|c|c|}
	\hline
	$A$ & $|V(X_A)| $ & Approximate size of $H_{d - 1}(X_A)_T$ \\ \hline
	$\{0, 1, 3\}$ & 41 & $2.157 \times 10^{33}$ \\
	$\{0, 2, 7\}$ & 43 & $1.205 \times 10^{63}$ \\
	$\{0, 6, 21\}$ & 53 & $1.972 \times 10^{84}$ \\ \hline
	$\{0, 2, 3, 4\}$ & 19 & $2.758 \times 10^{29}$ \\
	$\{0, 1, 3, 4\}$ & 23 & $4.493 \times 10^{38}$ \\
	$\{0, 1, 5, 11\}$ & 29 & $3.730 \times 10^{253}$ \\ \hline
	$\{0, 1, 3,4, 5\}$ & 13 & $4.118 \times 10^{16}$ \\
	$\{0, 2, 7, 8, 9\}$ & 17 & $4.011 \times 10^{102}$ \\
	$\{0, 1, 2, 3, 6\}$ & 19 & $2.377 \times 10^{150}$ \\
	\hline
  \end{tabular}
  \caption{Examples of torsion groups in homology of sum complexes}\label{tbl:sumcomplexes}
\end{table}

A second source of interesting examples comes from the torsion burst in the Linial--Meshulam model of random simplicial complexes. The torsion burst in the Linial--Meshulam refers to the apparent emergence of torsion in the codimension-1 homology group immediately before the first nontrivial cycle appears in top homology, that is around the threshold found in \cite{AL,LP}. For example, computational experiments examining the torsion burst in the Linial--Meshulam model in \cite{KLNP} found a 5-dimensional simplicial complex $X$ on 16 vertices with
$$H_4(X) \cong \Z^{36} \oplus \Z/1{,}147{,}712{,}621{,}067{,}945{,}810{,}235{,}354{,}141{,}226{,}409{,}657{,}574{,}376{,}675\Z.$$

This phenomenon of enormous torsion in homology in this random setting has been observed experimentally, for example by \cite{LP2} and by \cite{KLNP} but the reason it occurs remains unknown. Nevertheless, Table \ref{tbl:torsionburst} provides examples of randomly generated simplicial complexes with torsion in homology coming from the Linial--Meshulam torsion burst. For more background on the torsion burst, see \cite{KLNP}. \\

\begin{table}[H]
  \centering
  \renewcommand{\arraystretch}{1.2}
  \begin{tabular}{|c|c|c|}
	\hline
	$d$ & $n $ & Approximate size of $H_{d - 1}(X)_T$ \\ \hline
	2 & 50 & 27288 \\
	2 & 100 & $9.236 \times 10^{58}$ \\
	2 & 150  & $6.691 \times 10^{205}$ \\
	2 & 200 & $3.102 \times 10^{406}$ \\ \hline
	3 & 20 & 516194 \\
	3 & 30 & $8.503 \times 10^{82}$ \\
	3 & 40 & $7.832 \times 10^{294}$ \\
	3 & 50 & $3.423 \times 10^{722}$ \\ \hline
	4 & 15 & 4464 \\
	4 & 20 & $3.172 \times 10^{94}$ \\
	4 & 25 & $3.099 \times 10^{388}$ \\ \hline
	5 & 14 & 35162606 \\
	5 & 17 & $7.521 \times 10^{82}$ \\
	5 & 20 & $2.451 \times 10^{389}$\\
	\hline
  \end{tabular}
  \caption{Examples of torsion groups in homology from the torsion burst of random complexes}\label{tbl:torsionburst}
\end{table}

While \cite{KLNP,Kalai,LMR} provide results and examples establishing that small complexes can have large torsion in homology, our purpose here is to answer an inverse question: for a finite abelian group $G$ and dimension $d \geq 2$, how many vertices are necessary to construct a $d$-dimensional simplicial complex $X$ so that $H_{d - 1}(X)_T$ is isomorphic to $G$? Towards answering this question, we define for $d \geq 2$ and $G$ a finite abelian group, $T_d(G)$ as the minimal number of vertices $n$ so that there is a simplicial complex $X$ on $n$ vertices with the torsion part of $H_{d - 1}(X)$ isomorphic to $G$.\\

 Some results on triangulating projective space, see for example \cite{Kuhnel,Lutz,Walkup}, may be used to provide upper bounds on $T_d(\Z/2\Z)$ for any dimension $d$. Additionally, for $d = 2$ and $m \in \N$, one may use a ``repeated squares" presentation of $\Z/m\Z$ to show that $T_2(\Z/m\Z) = O(\log m)$. Such a construction is described by David Speyer on a MathOverflow post \cite{Speyer} responding to a question of John Palmieri \cite{Palmieri}. Speyer's construction provides inspiration for the first part of our construction used in the proof of our main theorem. Our main theorem is the following:
\begin{theorem}
For every $d \geq 2$, there exist constants $c_d$ and $C_d$ so that for any finite abelian group $G$, $$c_d(\log |G|)^{1/d} \leq T_d(G) \leq C_d(\log |G|)^{1/d}.$$
\end{theorem}
The lower bound is already known. Indeed it is given by the following theorem, which appears to be first due to \cite{Kalai} but a result like this also appears in \cite{Soule} who attributes it to Gabber. A proof may be found in, for example, \cite{Kalai} or \cite{HKP2}. 
\begin{unnumtheorem}[Theorem 4, part 1 of \cite{Kalai}]
If $X$ is a $d$-dimensional simplicial complex on $n$ vertices then $|H_{d-1}(X)_T| \leq \sqrt{d+1}^{\binom{n - 2}{d}}$. 
\end{unnumtheorem}

We should note here that \cite{Kalai} states the above theorem for $X$ a $d$-dimensional $\Q$-acyclic complex on $n$ vertices, however it can be checked that this implies the result over all $d$-dimensional simplicial complexes on $n$ vertices, see for example \cite{HKP2}. \\

\section{Overview of the proof} \label{overview}

Let $X$ be a finite simplicial complexes. Denote by $\Delta_{i, j}(X)$ the maximum degree of an $i$-dimensional face in $j$-dimensional faces, that is 
$$\Delta_{i, j}(X) = \max_{\sigma \in \skel_i(X)} |\{\tau \in \skel_j(X) \mid \sigma \subseteq \tau \}|,$$ 
where $\skel_k(X)$ denotes the set of $k$-dimensional faces of $X$.
We denote by $\Delta(X)$ the maximum over $i$ and $j$ of $\Delta_{i, j}(X)$. Throughout the proof of Theorem 1, it will be important that $\Delta(X)$ is bounded, for various simplicial complexes $X$. Of course, if $X$ is a simplicial complex and $\Delta_{0, 1}(X)$ is bounded by some constant then $\Delta(X)$ is bounded by a constant depending on $\Delta_{0, 1}(X)$. Nevertheless, in the interest of simplifying statements and proofs, it is convenient to have the notation $\Delta(X)$ and a single bound for it.\\

With this notation in hand, we are ready to give an outline of the proof of the main theorem. The goal of the paper will be to provide a construction, given a dimension $d \geq 2$, which proves that the upper bound in the statement of Theorem 1 is correct. The first step will be to prove the following lemma.
\begin{lemma}\label{general}
For every $d \geq 2$, there exists a constant $K$ depending only on $d$ so that for every finite abelian group $G$ there is a $d$-dimensional simplicial complex $X$ on at most $K \log_2 |G|$ vertices with $\Delta(X) \leq K - 1$ and $H_{d - 1}(X)_T$ isomorphic to $G$.
\end{lemma}
Of course this lemma alone does not prove the upper bound in Theorem 1, but using this initial construction we will build a smaller complex which does. Towards explaining this ``reduction step" we introduce the following definition.
\begin{definition}
If $X$ is a simplicial complex with a coloring $c$ of $V(X)$ we define the \emph{pattern} of a face to be the multiset of colors on its vertices. If $c$ is a proper coloring, in the sense that no two vertices connected by an edge receive the same color, we define the \emph{pattern complex} $(X, c)$ to be the simplicial complex on the set of colors of $c$ so that a subset $S$ of the colors of $c$ is a face of $(X, c)$ if and only if there is a face of $X$ with $S$ as its pattern. It is easy to see that this is a simplicial complex. Indeed if $S$ is a set of colors which is a pattern for some face $\sigma$ of $X$ then for any $S' \subseteq S$, $S'$ is a pattern for some face of $\sigma$.
\end{definition}

The relevant fact about the pattern complex $(X, c)$ that we will use is the following lemma:

\begin{lemma}\label{reduce}
If $X$ is a $d$-dimensional simplicial complex and $c$ is a proper coloring of the vertices of $X$ so that no two $(d-1)$-dimensional faces of $X$ have the same pattern then the complex $(X, c)$ has $H_{d - 1}(X)_T \cong H_{d - 1}((X, c))_T$.
\end{lemma}
\begin{proof}
Suppose that $V(X)$ is colored properly by $c$ with no two $(d-1)$-dimensional faces receiving the same pattern. We can define a simplicial map $f : X \rightarrow (X, c)$ sending each vertex $v$ to $c(v)$. Since no two $(d-1)$-dimensional faces receive the same pattern, we also have that no two $d$-dimensional faces receive the same pattern. Therefore $f$ induces a homeomorphism from $X/X^{(d - 2)}$ to $(X, c)/(X, c)^{(d - 2)}$. Now, taking the quotient of a $d$-dimensional CW-complex by its $(d - 2)$-skeleton preserves the torsion part of the $(d - 1)$st homology group. (This follows, for example, from theorem 2.13 from \cite{Hatcher}.) Thus $H_{d-1}(X)_T \cong H_{d - 1}(X/X^{(d - 2)})_T \cong H_{d - 1}((X, c)/(X, c)^{(d - 2)})_T \cong H_{d - 1}((X, c))_T$.
 \end{proof}

Thus the second step in the proof of Theorem 1 will be to show that there is a coloring of the vertices of the initial construction in a way that the pattern complex will be the final construction that we want. This is accomplished using the probabilistic method in proving the following lemma in Section \ref{final}.

\begin{lemma}\label{colorings}
Let $X$ be a $d$-dimensional simplicial complex, for $d \geq 2$, on $n$ vertices with $\Delta(X)\leq K - 1$ for some integer $K \geq 5$, then there exists a proper coloring $c$ of $V(X)$ having at most $18K^8d^6\sqrt[d]{n}$ colors so that no two $(d-1)$-dimensional faces of $X$ receive the same pattern by $c$.
\end{lemma}

Now assuming Lemmas \ref{general} and \ref{colorings} we give the proof of the upper bound in Theorem 1.
\begin{proof}[Proof of the upper bound in Theorem 1]
Fix $d \geq 2$, and let $G$ be a finite abelian group. By Lemma \ref{general} there exists a constant $K \geq 5$ depending only on $d$ and a simplicial complex $X$ with $H_{d - 1}(X)_T \cong G$, $\Delta(X) \leq K - 1$, and $|V(X)| \leq K \log_2 |G|$. Now by Lemma \ref{colorings}, there is a coloring $c$ of the vertices of $X$ with at most $18K^8d^6 \sqrt[d]{K \log_2 |G|}$ colors so that no two $(d - 1)$-dimensional faces of $X$ receive the same pattern by $c$. Therefore by Lemma \ref{reduce}, $H_{d - 1}((X, c))_T \cong G$, and so
$$T_d(G) \leq |V((X, c))| \leq \frac{18K^{8 + d^{-1}} d^6}{\sqrt[d]{\log 2}} \sqrt[d]{\log |G|},$$
proving Theorem 1 with $C_d =  18K^{8 + d^{-1}} d^6/\sqrt[d]{\log 2}$.
 \end{proof}
\section{The Initial Construction} \label{init}
It is easy to see that Lemma \ref{general}, the first step in our proof of Theorem 1, is implied by the following special case:

\begin{lemma} \label{const}
For every $d \geq 2$ there exists a constant $K$ depending only on $d$ so that for every integer $m \geq 2$ there is a $d$-dimensional simplicial complex $X$ on at most $K \log_2 m$ vertices with $\Delta(X) \leq K-1$ and  $H_{d - 1}(X)_T$ isomorphic to $\Z/m\Z$.
\end{lemma}

\begin{proof}[Proof of Lemma \ref{general} from Lemma \ref{const}]
Fix $d$ and let $G = \Z/m_1 \Z \oplus \Z/m_2 \Z \oplus \cdots \oplus \Z / m_l \Z$ with $m_1 | m_2 | \cdots | m_l$ be an arbitrary finite abelian group. By Lemma \ref{const}, there is a constant $K$ so that for each $i \in [l]$ there exists $X_i$ so that $H_{d - 1}(X_i)_T \cong \Z/m_i \Z$ with $\Delta(X_i) \leq K-1$ and $|V(X_i)| \leq K \log_2(m_i)$. Let $X$ be the disjoint union of all the $X_i$. Clearly, $\Delta(X) \leq K-1$, $H_{d - 1}(X)_T \cong G$, and $|V(X)| \leq \sum_{i = 1}^l K \log_2(m_i) = K \log_2(m_1m_2 \cdots m_l) = K \log(|G|).$
 \end{proof}

The main purpose of this section and Section \ref{gencasesection} will be to prove Lemma \ref{const}. We will prove this theorem by giving an explicit construction which we call the sphere-and-telescope construction. The idea is to construct a space with a ``repeated squares presentation" of $\Z/m\Z$ as $H_{d - 1}(X)_T$. Given $m$, write its binary expansion as $m = 2^{n_1} + \cdots + 2^{n_k}$ with $0 \leq n_1 < n_2 < \cdots < n_k$, then $\Z/m\Z$ is given by the abelian group presentation 
$$\langle \gamma_0, \gamma_1, ..., \gamma_{n_k} \mid 2\gamma_0 = \gamma_1, 2\gamma_1 = \gamma_2, ..., 2\gamma_{n_k - 1} = \gamma_{n_k}, \gamma_{n_1} + \gamma_{n_2} + \cdots + \gamma_{n_k} = 0\rangle.$$ The goal is to construct a simplicial complex with this presentation as the presentation for $H_{d - 1}(X)_T$ so that each $\gamma_i$ is a homology class of $H_{d - 1}(X)$ represented by the boundary of a $d$-simplex $z_i$. This will be accomplished by constructing two simplicial complexes $Y_1$ and $Y_2$ and attaching them to one another to build $X$. The required properties of $Y_1$ and $Y_2$ will be that $H_{d - 1}(Y_1) \cong \langle \gamma_0, \gamma_1, ..., \gamma_{n_k} \mid 2 \gamma_0 = \gamma_1, ..., 2 \gamma_{n_k - 1}  = \gamma_{n_k} \rangle$, $H_{d - 1}(Y_2) \cong \langle \tau_1, \tau_2, ..., \tau_k \mid \tau_1 + \tau_2 + \cdots + \tau_k = 0 \rangle$, and that $Y_1$ and $Y_2$ may be attached to one another in such a way that at the level of $(d-1)$st homology $\gamma_{n_i}$ is identified to $\tau_i$ for all $i \in \{1, 2, ..., k\}$.\\

To illustrate the idea of the construction, Figure \ref{fig:space1} shows that topological space that we would build in the case that $m = 25$ and $d = 2$. The full triangulation is omitted, but it is the space we construct up to homeomorphism. On the righthand side of the figure we have the telescope portion of the construction. Each segment is a punctured projective plane, or equivalently the mapping cylinder for the degree 2 map from $S^1$ to $S^1$. With the labeling on each copy of the punctured projective plane, at the level of homology we have the relators $2 \gamma_0 = \gamma_1, 2\gamma_1 = \gamma_2, 2 \gamma_2 = \gamma_3$, and $2\gamma_3 = \gamma_4$ where $\gamma_0, ..., \gamma_4$ are homology classes each represented by an $S^1$. The lefthand side is the sphere portion of the construction, though in reality it is a multipunctured sphere. With the labeling on the cycles we have that the first homology group of this space is given by $\langle \gamma_0, \gamma_1, \gamma_3 \mid \gamma_0 + \gamma_3 + \gamma_4 = 0 \rangle$. According to the identifications of different copies of $S^1$ in the figure we get that the torsion part of the first homology group for this space is given by $\langle \gamma_0, ..., \gamma_4 \mid 2\gamma_0 = \gamma_1, ..., 2\gamma_3 = \gamma_4, \gamma_0 + \gamma_3 + \gamma_4 = 0 \rangle$ thus we have that the homology class $\gamma_i = 2^i \gamma_0$, and since $25 = 2^0 + 2^3 + 2^4$ we have that the torsion part of the homology group is $\Z/25\Z$. It is also worth pointing out that we do get three free homology classes by how we attach the segments of the telescope as handles to the sphere.

\begin{center}
\begin{figure}[h]
\centering
\tikzstyle{every node}=[]
\begin{centering}                        
\begin{tikzpicture}[thick,scale=1.5]
\begin{scope}[xshift = -2 cm, scale = 0.9]
\shade[ball color = gray, opacity = 0.8] (0, 0) circle (2cm);
\draw (0, 0) circle (2cm);
\draw(-2, 0) arc(180:360:2 and 0.5);
\draw(2, 0) [dashed] arc(0: 180: 2 and 0.5);
\draw[fill = white, opacity = 0.8](0.4, 1.6) arc(0:360:0.4 and 0.25);
\draw[fill = white, opacity = 0.8](-0.6, -1.4) arc(-60:-420:0.35 and 0.25);
\draw[fill = white, opacity = 0.8](0.8, -1) arc(0:-360:0.3 and 0.25);

\tikzstyle{every node} = [circle, draw, fill = black, inner sep = 0, minimum width = 4];
\draw[<-, opacity  = 1](0.4, 1.6) arc(0:360:0.4 and 0.25) node[midway]{};
\draw[<-, opacity = 1](-0.6, -1.4) arc(-60:-420:0.35 and 0.25) node[midway]{};
\draw[<-, opacity = 1](0.8, -1) arc(0:-360:0.3 and 0.25) node[midway]{};

\tikzstyle{every node}=[]
\draw (0.5, 1.35) node {$\gamma_0$};
\draw(-0.35, -1.5) node {$\gamma_3$};
\draw(0.9, -0.9) node {$\gamma_4$};
\end{scope}

\begin{scope}[xshift = 1 cm]

\tikzstyle{every node} = [circle, draw, fill = black, inner sep = 0, minimum width = 4];
\draw[fill = gray, opacity = 0.8] (0, 1.5) circle (0.4 cm);
\draw[fill = gray, opacity = 0.8] (0, 0.5) circle (0.4 cm);
\draw[fill = gray, opacity = 0.8] (0, -0.5) circle (0.4 cm);
\draw[fill = gray, opacity = 0.8] (0, -1.5) circle (0.4 cm);

\draw[fill = white, opacity = 1] (0, 1.5) circle (0.2 cm);
\draw[fill = white, opacity = 1] (0, 0.5) circle (0.2 cm);
\draw[fill = white, opacity = 1] (0, -0.5) circle (0.2 cm);
\draw[fill = white, opacity = 1] (0, -1.5) circle (0.2 cm);

\draw [->, opacity = 1] (-0.4, 1.5) arc(180:360:0.4) node [midway] {};
\draw [->, opacity = 1] (0.4, 1.5) arc(0:180:0.4) node [midway] {};

\draw [->, opacity = 1] (-0.4, 0.5) arc(180:360:0.4) node [midway] {};
\draw [->, opacity = 1] (0.4, 0.5) arc(0:180:0.4) node [midway] {};

\draw [->, opacity = 1] (-0.4, -0.5) arc(180:360:0.4) node [midway] {};
\draw [->, opacity = 1] (0.4, -0.5) arc(0:180:0.4) node [midway] {};

\draw [->, opacity = 1] (-0.4, -1.5) arc(180:360:0.4) node [midway] {};
\draw [->, opacity = 1] (0.4, -1.5) arc(0:180:0.4) node [midway] {};

\draw [->, opacity = 1] (0, 1.3) arc(-90:90:0.2);
\draw [->, opacity = 1] (0, 0.3) arc(-90:90:0.2);
\draw [->, opacity = 1] (0, -0.7) arc(-90:90:0.2);
\draw [->, opacity = 1] (0, -1.7) arc(-90:90:0.2);

\draw (0, 1.3) node {};
\draw (0, 0.3) node {};
\draw (0, -0.7) node {};
\draw (0, -1.7) node {};

\tikzstyle{every node}=[]
\draw (0.6, 1.5) node{$\gamma_0$};
\draw (-0.6, 1.5) node{$\gamma_0$};

\draw (0.6, 0.5) node{$\gamma_1$};
\draw (-0.6, 0.5) node{$\gamma_1$};

\draw (0.6, -1.5) node{$\gamma_3$};
\draw (-0.6, -1.5) node{$\gamma_3$};

\draw (0.6, -0.5) node{$\gamma_2$};
\draw (-0.6, -0.5) node{$\gamma_2$};

\draw (0, 1.5) node{$\gamma_1$};
\draw (0, 0.5) node{$\gamma_2$};
\draw (0, -0.5) node {$\gamma_3$};
\draw (0, -1.5) node {$\gamma_4$};
\end{scope}
\end{tikzpicture}
\caption{The topological space which we triangulate in our construction for $m = 25$}\label{fig:space1}
\end{centering}
\end{figure}
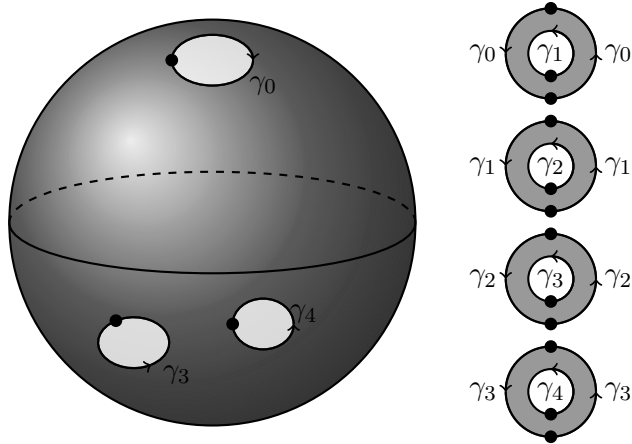
\end{center}

In higher dimensions the idea is exactly the same, but we increase the dimension of the building blocks appropriately. The new free homology classes created by attaching the telescope to the sphere will always occur in $H_1$, and so in fact in higher dimensions the $(d-1)$st homology group will be exactly the torsion group we want; there will be no free part.\\

In this section we prove Lemma \ref{const} in the $d = 2$ case. The full details for the construction in an arbitrary number of dimensions is given as Section \ref{gencasesection} however the $d = 2$ case illustrates the idea without having to get into all of the more technical details that are necessary in higher dimensions. Furthermore, a good understanding of the $d = 2$ case helps in understanding the general case.

\begin{unnumlemma}[Statement of Lemma \ref{const} for $d = 2$]\label{2case}
For every integer $m \geq 2$, there exists a simplicial complex $X$ so that $\Delta(X) \leq 34$, $|V(X)| \leq 50 \log_2 m$, and $H_1(X)_T \cong \Z/m\Z$. 
\end{unnumlemma}
\begin{proof}
Let $m$ be given. Write $m$ in its binary expansion $m = 2^{n_1} + 2^{n_2} + \cdots + 2^{n_k}$ with $0 \leq n_1 < n_2 < \cdots < n_k$. We first build the telescope portion of the construction. This is accomplished by attaching several copies of the standard triangulation of the punctured projective plane end-to-end (The triangulation is the one obtained by taking the triangulation of the projective plane obtained from antipodal identification on the icosahedron and removing a single face.) Explicitly the telescope portion of the construction, denoted $Y_1$, is the 2-dimensional simplicial complex with vertex set $v_0, v_1, v_2, v_3,, ..., v_{3n_k + 2}$ having as its facets:
\begin{eqnarray*}
\skel_2(Y_1) &=& \{[v_{3i}, v_{3i + 1}, v_{3i + 4}], [v_{3i + 1}, v_{3i + 2}, v_{3i + 4}], [v_{3i + 2}, v_{3i + 4}, v_{3i + 5}], [v_{3i}, v_{3i + 2}, v_{3i + 5}], \\
&& [v_{3i}, v_{3i + 1}, v_{3i + 5}], [v_{3i + 1}, v_{3i + 3}, v_{3i + 5}], [v_{3i + 1}, v_{3i + 2}, v_{3i + 3}], [v_{3i}, v_{3i + 2}, v_{3i + 3}], \\
&& [v_{3i}, v_{3i + 3}, v_{3i + 4}] \mid i = 0, 1, 2, ..., (n_k - 1)\}
\end{eqnarray*}
The faces of $Y_1$ are easier to see from a picture. Figure \ref{fig:projplane} shows a ``building block" of $Y_1$, the triangulated projective plane with a face removed. The full set of faces of $Y_1$ are the faces in Figure \ref{fig:projplane} as $i$ ranges over $\{0, 1, 2, ..., (n_k - 1)\}$. 
\begin{center}
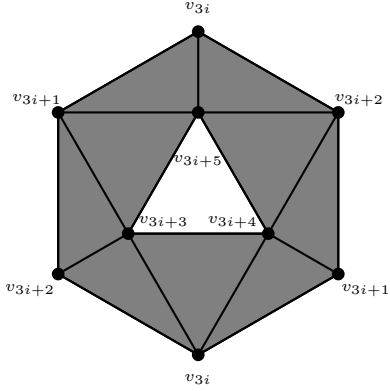
\begin{figure}[h]
\centering
\tikzstyle{every node}=[circle, draw, fill=black,
                        inner sep=0pt, minimum width=4pt]
\begin{centering}                        
\begin{tikzpicture}[thick,scale=2.15]
\draw[fill = gray] (90: 1) -- (150: 1) -- (210:1) -- (270:1) -- (330:1) -- (390: 1) -- cycle;
\draw[fill = white] (90:0.5) -- (210:0.5) -- (330:0.5) -- cycle;
\foreach \x in {90, 150, ..., 390}
{
\draw (\x:1)--(\x+60:1);
}

\foreach \x in {90, 210, 330}
{
	\draw (\x:0.5) -- (\x + 120:0.5);
	\draw (\x: 0.5) -- (\x: 1);
}

\foreach \x in {90, 210, 330}
{
	\draw(\x:0.5) -- (\x + 60: 1);
	\draw(\x:0.5) -- (\x - 60: 1);
}

\foreach \x in {90, 210, 330}
{
	\draw (\x:0.5) node{};
}
\foreach \x in {90, 150, ..., 390}
{
\draw(\x:1) node {};
}
\tikzstyle{every node}=[]

{
\tiny
\draw (90:1.15) node {$v_{3i}$};
\draw(150:1.15) node{$v_{3i + 1}$};
\draw(210:1.20) node{$v_{3i + 2}$};
\draw (270:1.15) node {$v_{3i}$};
\draw(330:1.20) node{$v_{3i + 1}$};
\draw(390:1.15) node{$v_{3i + 2}$};
}

{
\tiny
\draw (90:0.2) node {$v_{3i + 5}$};
\draw(220:0.28) node{$v_{3i + 3}$};
\draw(320:0.28) node{$v_{3i + 4}$};

}


\end{tikzpicture} 
\end{centering}
\caption{The building block for the telescope construction}\label{fig:projplane}
\end{figure}
\end{center}
If we order the vertices according to their natural ordering and let that ordering induce an orientation on all the edges and faces of $Y_1$ then letting $\gamma_i$ denote the 1-cycle of $Y_1$ represented by $[v_{3i}, v_{3i + 1}] - [v_{3i}, v_{3i + 2}] + [v_{3i + 1}, v_{3i + 2}]$ for $i \in \{0, 1, 2, ..., n_k\}$, we have that $2 \gamma_i - \gamma_{i + 1}$ is a 1-boundary of $Y_1$ for all $i \in \{0, ..., n_{k} - 1\}$. Now $H_1(Y_1)$ can be presented as $\langle \gamma_0, \gamma_1, ..., \gamma_{n_k} \mid 2\gamma_0 = \gamma_1, 2\gamma_1 = \gamma_2, ..., 2\gamma_{n_k - 1} = \gamma_{n_k} \rangle$. (There is a bit to do to check this; one way to check is by induction on $n_k$, if $n_k = 1$, then $Y_1$ is just the triangulated projective plane with 1 face removed. The inductive step follows from the Mayer--Vietoris sequence; full details for this are provided in Section \ref{Y1details}.) Note that $\Delta(Y_1) \leq 10$ and $|V(Y_1)| = 3(n_k + 1)$.\\

Next we construct a complex $Y_2$ which we will attach to $Y_1$ in a certain way to build our complex $X$. As the name of the construction and Figure \ref{fig:space1} suggest, $Y_2$ will be the ``sphere" portion of our construction. Indeed we will construct $Y_2$ by simply removing the right number of faces from a certain triangulation of $S^2$. However, we have to keep the degree of the complex bounded regardless of $m$, so we want to choose our triangulation of $S^2$ carefully. Toward that we use the following lemma to get a triangulation of $S^2$ from which we delete faces to obtain $Y_2$. We prove this as Lemma \ref{sphere} in Section \ref{Y2details}.

\begin{lemma}\label{2sphere}
For every $k \in \N$ there exists a triangulation $T = T(k)$ of $S^2$ so that $\Delta_{0, 1}(T) \leq 24$, $V(T) = 7k + 16$, and $T$ has $k$ 2-dimensional faces which are vertex-disjoint and have no edges from the vertices of one to the vertices of another.
\end{lemma}

To build $Y_2$, start by using Lemma \ref{2sphere} to get a triangulation of $S^2$, $T$ with $\Delta(T) \leq 24$ (this follows from the statement since $\Delta_{0, 1}(T) \leq 24$ and every edge is contained in exactly two faces), $|V(T)| \leq 7(2k) + 16$, and $2k$ 2-dimensional faces, $t_1$, $t_2$, ..., $t_{2k}$ which are vertex disjoint from one another and have no edges from the vertices of one to the vertices of another. Now assign an ordering to the vertices of $T$ and give the faces of $T$ the orientation induced by this ordering. Since $T$ is a triangulated 2-sphere there is a 2-chain $(x_1, ..., x_l)$ (where $l$ is the number of 2-dimensional faces of $T$) so that $|x_i| = 1$ for all $i$ and so that $\partial_2(x_1, ..., x_l) = 0$ where $\partial_2$ denotes the top-dimensional boundary matrix of $T$. Now without loss of generality at least $k$ of the faces $t_1,t_2, ..., t_{2k}$ have a coefficient of $-1$ in the 2-chain $x$. It follows that $k$ of these faces may be removed to create an oriented simplicial complex, which we call $Y_2$, which has $H_1(Y_2) = \langle \tau_1, \tau_2, ..., \tau_k \mid \tau_1 + \tau_2 + \cdots + \tau_k = 0 \rangle$ where each $\tau_i$ represents the positively-oriented boundary of a removed face. For each $i$, let $w_{3i}, w_{3i + 1}, w_{3i + 2}$ denote the vertices of the face boundary representing $\tau_i$. That is, $\tau_i$ is represented by the 1-cycle $[w_{3i}, w_{3i + 1}] - [w_{3i}, w_{3i + 2}] + [w_{3i + 1}, w_{3i + 2}]$ where $w_{3i} < w_{3i + 1} < w_{3i + 2}$ in the vertex ordering on $Y_2$.\\

Now $Y_1$ and $Y_2$ will be attached together in a particular way to build the complex $X$ which has $H_1(X)_T \cong \Z/m\Z$. Let $S$ denote the subcomplex of $Y_2$ induced by $w_3$, $w_4$, $w_5$, $w_6$, $...$, $w_{3k}$, $w_{3k + 1}$, $w_{3k + 2}$. Since the faces we deleted from $T$ to build $Y_2$ are vertex-disjoint and have no edges between any two of them, $S$ is a disjoint union of $k$ triangle boundaries. Let $f: S \rightarrow Y_1$ be the simplicial map defined by $w_{3i} \mapsto v_{3n_i}$, $w_{3i + 1} \mapsto v_{3n_i + 1}$, and $w_{3i + 2} \mapsto v_{3n_i + 2}$. Now let $X = Y_1 \sqcup_f Y_2$, that is, $Y_1$ with $Y_2$ attached along $S$ via $f$ (this is defined as a topological space in Chapter 0 of \cite{Hatcher} as the the quotient of the disjoint union of $Y_1$ and $Y_2$ by attaching each point $s \in S$ to its image $f(s)$ in $Y_1$). Since $f$ is injective and $S$ is an induced subcomplex of $Y_2$, $X$ is a simplicial complex (this is proved as Lemma \ref{attaching}). \\
 
 Now we use the Mayer--Vietoris sequence to show that $H_1(X) = \Z^{k - 1} \oplus \Z/m\Z$. From the Mayer--Vietoris sequence we have the following exact sequence:

\begin{center}
$\begin{CD}
H_{1}(S) @>h>>  H_1(Y_2) \oplus H_1(Y_1) @>g>> H_1(X) @>>>\tilde{H}_{0}(S) @>>> 0 \\
\end{CD}$\\
\end{center}
We claim that $(H_1(Y_1) \oplus H_1(Y_2)) / \Im(h) \cong \Z/m$. This follows since $f$ has the effect of identifying the 1-cycle $[w_{3i}, w_{3i + 1}] - [w_{3i}, w_{3i + 2}] + [w_{3i + 1}, w_{3i + 2}]$ to the 1-cycle $[v_{3n_i}, v_{3n_i + 1}] - [v_{3n_i}, v_{3n_i + 2}] + [v_{3n_i + 1}, v_{3n_i + 2}]$, that is $f$ identifies $\tau_i$ to $\gamma_{n_i}$. It follows that the image of $h$ is $\langle(\tau_i, -\gamma_{n_i})_{i = 1}^k\rangle$ therefore
\begin{eqnarray*}
(H_1(Y_1) \oplus H_1(Y_2)) / \Im(h) &\cong& \langle \gamma_0, \gamma_1, ..., \gamma_{n_k}, \tau_1, ..., \tau_k \mid 2\gamma_0 = \gamma_1, 2\gamma_1 = \gamma_2, ..., 2\gamma_{n_k - 1} = \gamma_{n_k}, \\
&&\tau_1 + \tau_2 + \cdots + \tau_k = 0, \gamma_{n_1} = \tau_1, ..., \gamma_{n_k} = \tau_k \rangle \\
&\cong& \langle \gamma_0, \gamma_1, ..., \gamma_{n_k} \mid 2\gamma_0 = \gamma_1, ..., 2\gamma_{n_k - 1} = \gamma_{n_k}, \gamma_{n_1} + \gamma_{n_2} + \cdots + \gamma_{n_k} = 0 \rangle \\
&\cong& \langle \gamma_0 | (2^{n_1} + 2^{n_2} + \cdots + 2^{n_k})\gamma_0 = 0 \rangle \\
&\cong& \Z/m\Z.\\
\end{eqnarray*}
Therefore the the image of $g$ is isomorphic to $\Z/m\Z$ and since $\tilde{H}_0(S)$ is free of rank $k - 1$, exactness implies that $H_1(X) \cong \Z^{k - 1} \oplus \Z/m\Z$.\\

Finally $\Delta(X) \leq \Delta(Y_1) + \Delta(Y_2) \leq 10 + 24 = 34$. Also, $k \leq \log_2 m + 1$ and $n_k \leq \log_2 m$, and therefore 
$|V(X)| \leq |V(Y_1)| + |V(Y_2)| \leq 3(n_k + 1) + 14k + 16 \leq 3(\log_2 m + 1) + 14(\log_2 m + 1) + 16 \leq 17 \log_2 m + 33 \leq 50 \log_2 m$. This completes the proof.
 \end{proof}


We credit a portion of the construction in the $d = 2$ case of Lemma \ref{const} to Speyer \cite{Speyer}. Indeed the construction in \cite{Speyer} includes the same telescope portion that we have here. However, the two constructions vary in how they add the relator $\gamma^{n_1} + \gamma^{n_2} + \cdots + \gamma^{n_k} = 0$ to the first homology group. We use a different method for our construction in order to guarantee bounded degree. \\

The construction to prove Lemma \ref{const} in higher dimensions is similar. Given $m = 2^{n_1} + \cdots + 2^{n_k}$ with $\Z/m\Z \cong \langle \gamma_0, \gamma_1, ... \gamma_{n_k} \mid 2\gamma_0 = \gamma_1, 2 \gamma_1 = \gamma_2, ..., 2 \gamma_{n_k - 1} = \gamma_{n_k}, \gamma_{n_1} + \cdots + \gamma_{n_k} = 0 \rangle$, our goal is a simplicial complex $X$ where each $\gamma_i$ is represented by a positively-oriented $d$-simplex boundary all of which are disjoint from one another, so that there is a $d$-chain with $2 \gamma_i = \gamma_{i + 1}$ as its boundary for each $i \in \{0, 1, ..., n_{k-1}\}$, and a $d$-chain with $\gamma_{n_1} + \cdots + \gamma_{n_k}$ as its boundary.

\section{The Final Construction}\label{final}
In this section we use the probabilistic method, in particular the Lov\'{a}sz Local Lemma, to show how to finish the proof of Theorem 1 by proving Lemma \ref{colorings}; we save the details of the initial construction when $d > 2$ for Section \ref{gencasesection}. We begin by stating the Lov\'{a}sz Local Lemma as it is stated in Chapter 5 of \cite{AS}.
\begin{lovaszlemma}[\cite{EL}]
Let $A_1, A_2, ..., A_n$ be events in an arbitrary probability space. Suppose that each event $A_i$ is mutually independent of all the other events $A_j$ but at most $t$, and that $\Pr[A_i] \leq p$ for all $1 \leq i \leq n$. If $ep(t+1) \leq 1$ then $\Pr[\bigwedge_{i = 1}^n \overline{A_i}] > 0$
\end{lovaszlemma}
We now prove Lemma \ref{colorings}. In proof, we implicitly treat $\sqrt[d]{n}$ as an integer, when really we mean $\lceil \sqrt[d]{n} \rceil$.

\begin{proof}[Proof of Lemma \ref{colorings}]
We will find three colorings $c_1, c_2, c_3$ of $V(X)$ so that $c_1$ is a proper coloring of $X^{(1)}$, $c_2$ has no pair of intersecting $(d-1)$-dimensional faces receiving the same pattern, and $c_3$ has no pair of disjoint $(d-1)$-dimensional faces receiving the same pattern. We then let $c = (c_1, c_2, c_3)$, and $c$ will have the required properties with $|c| = |c_1||c_2||c_3|$ and we will show that this is bounded by the value in the statement. Finding $c_1$ is easy; there is a proper coloring of $X^{(1)}$ by at most $K$ colors since the vertex degree is bounded above by $K - 1$, choose such a proper coloring for $c_1$. To show that $c_2$ and $c_3$ exist, we will use the Lov\'{a}sz Local Lemma.\\

 We first show that there exists a coloring $c_2$ on at most $3d^5K^5$ colors so that no pair of intersecting $(d-1)$-dimensional faces receive the same pattern by $c_2$. We may define a graph $H$ from $X$ by letting the vertices of $H$ be the $(d-1)$-dimensional faces of $X$ with $(\sigma, \tau) \in E(H)$ if and only if $\sigma \cap \tau \neq \emptyset$, a coloring of $V(X)$ induces a coloring of $H$ by the patterns on the $(d-1)$-dimensional faces. We wish to show that there is a coloring of $V(X)$ by at most $3d^5K^5$ colors which induces a proper coloring on $H$. Consider the probability space of colorings of $V(X)$ by coloring each vertex uniformly at random from among a set of $3d^5K^5$ colors. For a fixed $(d-1)$-dimensional face $\sigma$, let $A_{\sigma}$ denote the event that some neighbor of $\sigma$ in $H$ receives the same pattern as $\sigma$. \\

By the bounded degree condition on $X$, for a $(d-1)$-dimensional face $\sigma$ the number of faces $\tau$ so that $(\sigma, \tau) \in E(H)$ is at most $dK$. Now $\sigma$ and $\tau$ are most likely to receive the same pattern if they share $(d-1)$ vertices, and so the probability that $\sigma$ and $\tau$ receive the same pattern is at most $((d - 1)/(3d^5K^5))^{d - 1}$ (This is an upper bound on the probability that the same multiset of colors used for $\sigma$ are used for $\tau$ given that $\sigma$ and $\tau$ meet at a single vertex); this probability is bounded above by $1/(3d^5K^5)$. Thus the probability of $A_{\sigma}$ is bounded above by $dK/(3d^5K^5)$ by a union bound over the neighbors of $\sigma$. We also observe that if the distance between $\sigma$ and $\tau$ in $H$ is at least 4, then $A_{\sigma}$ and $A_{\tau}$ are mutually independent as they are on disjoint vertex sets. Indeed if $A_{\sigma}$ and $A_{\tau}$ are not independent then some vertex $v$ must appear both in a face $\tau'$ in the closed neighborhood of $\tau$ in $H$ and in some face $\sigma'$ in the closed neighborhood of $\sigma$ in $H$. Thus there is a path $\sigma, \sigma', \tau', \tau$ of length at most 3 from $\sigma$ to $\tau$. It follows that each $A_{\sigma}$ is mutually independent from all $A_{\tau}$ but at most $(dK)^4$. Thus we may apply the Lov\'{a}sz Local Lemma since $$e\frac{dK}{3d^5K^5}(d^4K^4 + 1) \leq 3d^5K^5/(3d^5 K^5) = 1.$$ It follows that there is a coloring so that no $A_{\sigma}$ holds, i.e. a coloring so that no two intersecting $(d-1)$-dimensional faces receive the same pattern. Let $c_2$ be one of these colorings. \\

We now handle the disjoint $(d-1)$-dimensional faces, again using the Lov\'{a}sz Local Lemma. Consider the probability space of colorings of $V(X)$ by coloring each vertex uniformly at random from among a set of $6K^2d \sqrt[d]{n}$ colors. For $\sigma$ and $\tau$ disjoint $(d - 1)$-dimensional faces, let $A_{(\sigma, \tau)}$ denote the event that $\sigma$ and $\tau$ receive the same pattern. We have that 
$$\Pr(A_{\sigma, \tau}) \leq \frac{d^d}{6^d K^{2d}d^d n} = \frac{1}{6^dK^{2d} n}.$$
As in the case for $\sigma \cap \tau \neq \emptyset$, the bound on the probability above comes from an upper bound on the probability that the same multiset of colors used for $\sigma$ are used for $\tau$. Now for $(\sigma, \tau)$ and $(\sigma', \tau')$, $A_{(\sigma, \tau)}$ and $A_{(\sigma', \tau')}$ are mutually independent if the two pairs are on disjoint vertex sets. It follows that for any $(\sigma, \tau)$ the number of $(\sigma', \tau')$ so that $A_{(\sigma, \tau)}$ is not disjoint from $A_{(\sigma', \tau')}$ is at most $2(dK (K/d) n)$ (pick one of the at most $dK$ faces adjacent to $\sigma$ for $\sigma'$ and then any other of the at most $(K/d) n$ faces for $\tau$, then reverse the roles of $\sigma$ and $\tau$). Thus the Lov\'{a}sz Local Lemma may be used since 
$$e \frac{1}{6^dK^{2d}n} (2K^2 n + 1) \leq \frac{6K^2 n}{6K^2 n} = 1.$$ 
It follows that there is a coloring by at most $6K^2d\sqrt[d]{n}$ colors so that no disjoint $(d-1)$-dimensional faces receive the same pattern, let $c_3$ be such a coloring. We let $c = (c_1, c_2, c_3)$ then $c$ is on at most $K(3d^5K^5)(6K^2d \sqrt[d]{n}) = 18K^8d^6 \sqrt[d]{n}$ colors and has the required properties.
 \end{proof}
\section{Full proof of Lemma \ref{const}}\label{gencasesection}
The goal of this section will be to prove Lemma \ref{const} for any dimension $d \geq 2$.  Given $m$, write its binary expansion as $m = 2^{n_1} + \cdots + 2^{n_k}$, then $\Z/m\Z$ is given by the abelian group presentation $\langle \gamma_0, \gamma_1, ..., \gamma_{n_k} \mid 2\gamma_0 = \gamma_1, 2\gamma_1 = \gamma_2,..., 2\gamma_{n_k - 1} = \gamma_{n_k}, \gamma_{n_1} + \gamma_{n_2} + \cdots + \gamma_{n_k} = 0\rangle$. The goal is to construct a simplicial complex $X$ with this presentation as the presentation for $H_{d - 1}(X)_T$ which has degree bounded by a constant $K -1$ depending only on $d$ and at most $K \log_2 m$ vertices. This is accomplished by the following series of steps for any dimension $d \geq 2$.
\begin{enumerate}
\item Show that there exists a simplicial complex $P = P(d)$ so that $H_{d - 1}(P) = \langle a, b \mid 2a = b \rangle$ with each of the homology classes $a$ and $b$ represented by an embedded, positively-oriented copy of $\partial \Delta^d$ which are vertex disjoint from one another.
\item Show that there is a constant $L = L(d)$ so that for any integer $k \geq 0$ there is a triangulation $T$ of $S^d$ so that $\Delta(T) \leq L$, $|V(T)| \leq Lk$, and with $k$ $d$-dimensional faces $t_1, ..., t_k$ which are vertex-disjoint and nonadjacent. Nonadjacent in this context means that for $i \neq j$, there are no edges between the vertices of $t_i$ and the vertices of $t_j$.
\item Prove Lemma \ref{const} for $K = \max\{2\Delta(P) + L + 1,2|V(P)| + 4L \}$ by giving an explicit construction for $m \geq 2$, with $m = 2^{n_1} + 2^{n_2} + \cdots + 2^{n_k}$, in the following series of steps:
\begin{enumerate}
\item Attach $n_k$ copies of $P$ together to create a complex $Y_1$ with $\Delta(Y_1) \leq 2\Delta(P)$ and $H_{d - 1}(Y_1) = \langle \gamma_0, \gamma_1, ..., \gamma_{n_k} \mid 2\gamma_0 = \gamma_1, 2\gamma_1 = \gamma_2, ..., 2\gamma_{n_k - 1} = \gamma_{n_k}\rangle$ where each $\gamma_i$ is represented by an embedded, positively-oriented copy of $\partial \Delta^d$, $Z_i$ which are all vertex disjoint from one another. We refer to $Y_1$ as the telescope part of the construction.
\item Use step 2 above to construct a complex $Y_2$, which is a triangulation of $S^d$ with $k$ vertex-disjoint, nonadjacent, $d$-dimensional faces removed, with $\Delta(Y_2) \leq L$ and $|V(Y_2)| \leq 2Lk$. This complex will have $H_{d - 1}(Y_2) = \langle \tau_1, \tau_2, ..., \tau_k \mid \tau_1 + \tau_2 + \cdots + \tau_k = 0 \rangle$ where each $\tau_i$ is represented by an embedded, positively-oriented copy of $\partial \Delta^d$, $Z_i'$ so that $Z_1', ..., Z_k'$ are vertex-disjoint and nonadjacent. We refer to $Y_2$ as the sphere part of the construction.
\item Attach $Y_1$ to $Y_2$ by attaching $Z_{n_i}$ to $Z_i'$ for every $i \in \{1, ..., k\}$ in a way that identifies $\gamma_{n_i}$ to $\tau_i$ so that we get a simplicial complex $X$ which has 
\begin{eqnarray*}
H_{d - 1}(X)_T &\cong& \langle \gamma_0, \gamma_1, ..., \gamma_{n_k}, \tau_1, ..., \tau_k \mid 2\gamma_0 = \gamma_1, 2\gamma_1 = \gamma_2, ..., 2\gamma_{n_k - 1} = \gamma_{n_k}, \\
&&\tau_1 + \tau_2 + \cdots + \tau_k = 0, \gamma_{n_1} = \tau_1, ..., \gamma_{n_k} = \tau_k \rangle \\
&\cong& \langle \gamma_0, \gamma_1, ..., \gamma_{n_k} \mid 2\gamma_0 = \gamma_1, ..., 2\gamma_{n_k - 1} = \gamma_{n_k}, \gamma_{n_1} + \gamma_{n_2} + \cdots + \gamma_{n_k} = 0 \rangle \\
&\cong& \Z/m\Z,\\
\end{eqnarray*}
 and for which $$\Delta(X) \leq \Delta(Y_1) + \Delta(Y_2) \leq 2\Delta(P) + L,$$ and 
$$V(X) \leq V(Y_1) + V(Y_2) \leq V(P)n_k + 2Lk \leq (2V(P) + 4L) \log_2m.$$
\end{enumerate}
\end{enumerate}

\subsection{The attaching maps}
As the steps above indicate the construction will be built in pieces which will be attached to one another in a way that the final complex $X$ will have $H_{d - 1}(X)_T \cong \Z/m\Z$. Of course we will need to be careful in how we attach simplicial complexes to one another so that the resulting space is still a simplicial complex. We will be attaching our building blocks to each other in a fairly standard way. Given two topological spaces $A$ and $B$, a subspace $S \subseteq B$, and a map $f: S \rightarrow A$, the space $A \sqcup_f B$ may be defined by taking the quotient of the disjoint union of $A$ and $B$ by the equivalence relation $s \sim f(s)$ for all $s \in S$. This type of attaching is called attaching $B$ to $A$ along $S$ via $f$ and is described in, for example, chapter 0 of \cite{Hatcher}. In the case where $A$ and $B$ are simplicial complexes, $S$ is a set of simplex boundaries contained in $B$, and $f$ is a simplicial map, attaching $B$ to $A$ along $S$ via $f$ is similar to a connected sum of simplicial complexes (see for example \cite{MatsumuraMoore}). In the case of attaching $Y_1$ to $Y_2$ our attachments will be similar to multiple handle additions (see for example \cite{Walkup}). Indeed, adding these handles creates free homology classes in $H_1$ in our construction as we mention in Section \ref{init}. \\

The main difference between our present setting and the usual simplicial connected sum of $A$ and $B$ along $S$, is that $S$ will be a collection of $d$-simplex boundaries, rather than a collection of $d$-simplicies whose interiors we delete after attachment, and that we must be careful with orientations in order to easily compute the homology of $A \sqcup_f B$ from that of $A$ and $B$. In the present setting, we will have that $A$ and $B$ are finite simplicial complexes, $S$ is an induced subcomplex of $B$, and $f : S \rightarrow A$ will be a simplicial injection, and thus a homeomorphism onto is image. We first verify that in this case $A \sqcup_f B$ remains a simplicial complex.
\begin{lemma}\label{attaching}
Let $A$ and $B$ be simplicial complexes, with $S_1$ a subcomplex of $A$ and $S_2$ an induced subcomplex of $B$, so that there is a simplicial homeomorphism $f: S_2 \rightarrow S_1$. Then $A \sqcup_f B$ is a simplicial complex.
\end{lemma}
\begin{proof}
As $f$ is a simplicial homeomorphism whose domain is an induced subcomplex of $B$ we have that $A \sqcup_f B$ may be realized as a pattern complex (as defined in Section 2 above) with respect to a certain coloring of $A \sqcup B$. First color every vertex of $B$ uniquely. Then color each $w \in S_1$ with the same color used for its unique preimage under $f$ in $S_2$. Finally color all the vertices in $A \setminus S_1$ uniquely. Now any two faces which receive the same pattern are identified together by $f$ since $S_2$ is an induced subcomplex and $f$ is determined by its image on vertices. Thus $A \sqcup_f B$ is the pattern complex of $A \sqcup B$ with respect to this coloring we have described. So it is a simplicial complex. 
 \end{proof}

The attaching maps used in the construction will always satisfy the assumptions of Lemma \ref{attaching}; thus at each step in the process the construction will be a simplicial complex. Furthermore, by attaching two simplicial complexes $A$ and $B$ in the way described in Lemma \ref{attaching} will result in a way to express $A \sqcup_f B$ as a union of two subspaces one of which is homotopy equivalent to $A$ and one of which is homotopy equivalent to $B$ with their intersection being homotopy equivalent to $S_2$. Thus we will simplify notation slightly and write $A \cup B$ for $A \sqcup_f B$ and $A \cap B$ for $S_2 \subseteq A \sqcup_f B$ when it is clear which simplicial homeomorphism $f: S_2 \rightarrow S_1$ we are using, especially as it relates to using the Mayer--Vietoris sequence to compute the homology of $A \sqcup_f B$ from the homology of $A$ and $B$. \\

While this process of attaching $A$ and $B$ is well-defined and gives us the perfect setting to use the Mayer--Vietoris sequence, there is still one issue: the orientation of the faces of $A \cup B$. To compute the homology of $A$ and the homology of $B$ we are required to choose an orientation on each face of $A$ and each face of $B$. When we attach $B$ to $A$ along $S_2$ by $f$ we get a new simplicial complex $A \cup B$. Therefore it is necessary to choose orientations of the faces of $A \cup B$ in order to compute its homology, and in particular it is necessary to make a decision for the orientation of faces of $A \cap B$. While the choice of orientations does not affect the homology groups of a simplicial complex up to group isomorphism, it may not be easy to compute the homology groups of $A \cup B$ from the homology groups of $A$ and the homology groups of $B$ if the orientations chosen for $A \cap B$ cannot be made to match the initial orientations assigned to those faces in $A$ and in $B$. Towards addressing this issue in the construction we give the following definition:
\begin{definition}
Let $A$ be an oriented simplicial complex. For $S$ a subcomplex of $A$, we say that $S$ is \textit{ coherently ordered with respect to the orientation on $A$} (or just \textit{coherently ordered} when $A$ and its orientation are clear from context) provided that there is an ordering $v_1, ..., v_k$ of the vertices of $S$ so that the orientation of each face of $S$ induced by this ordering on the vertices matches its orientation in $A$. We say that the ordering $v_1, ..., v_k$ is a \textit{coherent ordering of $S$}. 
\end{definition}


Primarily, this definition will be used when we attach $d$-dimensional simplicial complexes $A$ and $B$ to one another so that some element of $H_{d - 1}(A)$ is identified to some element of $H_{d - 1}(B)$ in a way that makes $H_{d - 1}(A \cup B)$ easy to compute using the Mayer--Vietoris sequence. Towards that goal we define the following:

\begin{definition}
Let $A$ be an oriented simplicial complex, and let $d$ be a positive integer. We say that $\gamma \in H_{d - 1}(A)$ is \textit{coherently represented by the $d$-simplex boundary $Z$} provided that $Z$ is an embedded $d$-simplex boundary in $A$ which is induced (i.e. the interior $d$-dimensional face is not in $A$), and which has a coherent ordering $v_0, ..., v_d$, so that $\gamma$ is homologous to the $(d-1)$-cycle:
$$\sum_{i = 0}^d (-1)^i [v_0, v_1, .., \hat{v_i}, ..., v_d].$$
\end{definition}

Now, all of the attachments between building blocks in the construction will be along simplex boundaries which coherently represent cycles in homology.  It is worth mentioning here that a coherent ordering on a $d$-simplex boundary is necessarily unique when $d \geq 2$ (and irrelevant if $d = 1$). Thus we will not encounter any issues with having to choose some coherent ordering. With this definition we are ready to state and prove the main attaching lemma:
\begin{lemma}\label{attachinglemma}
Let $d$ and $t$ be positive integers with $d \geq 2$. Let $A$ and $B$ be connected, oriented simplicial complexes. Let $\gamma_1, ..., \gamma_t \in H_{d-1}(A)$ be coherently represented by the $d$-simplex boundaries $Z_1, Z_2, ..., Z_t$ respectively with $Z_i \cap Z_j = \emptyset$ for all $i \neq j$. Let $\tau_1, ..., \tau_t \in H_{d-1}(B)$ be coherently represented by the $d$-simplex boundaries $Z_1', ..., Z_t'$ respectively with $Z_i' \cap Z_j' = \emptyset$ and no edges between the vertices of $Z_i'$ and the vertices of $Z_j'$ for every $i \neq j$. For each $i$, let $v_0^i, v_1^i, ..., v_d^i$ denote the coherent ordering for $Z_i$ and $w_0^i, w_1^i, ..., w_d^i$ denote the coherent ordering for $Z_i'$. Let $f : \bigsqcup_{i = 1}^{t} Z_i' \rightarrow \bigsqcup_{i = 1}^{t} Z_i$ be the simplicial map defined by $w_j^i \mapsto v_j^i$ for all $i \in \{1, .., t\}$ and $j \in \{0, ..., d\}$. Then $A \sqcup_f B$ is a simplicial complex and $$H_{d-1}(A \sqcup_f B) \cong (H_{d-1}(A) \oplus H_{d-1}(B)) / \langle (\gamma_i, -\tau_i)_{i \in [t]} \rangle \oplus \tilde{H}_{d - 2}\left(\bigsqcup_{i = 1}^{t} S^d\right).$$

\end{lemma}
\begin{proof}
First, $\bigsqcup_{i = 1}^t Z_i'$ is an induced subcomplex of $B$ since there are no edges between $Z_i'$ and $Z_j'$ for any $i \neq j$. Also $f$ is bijective as a simplicial map and so it is a simplicial homeomorphism, thus by Lemma \ref{attaching} $A \sqcup_f B$ is a simplicial complex. \\

Now we compute the homology. By construction, we have a portion of the Mayer--Vietoris sequence below for computing the homology of $A \cup B$ ($= A \sqcup_f B$) from the homology of $A$ and $B$.
\begin{center}
$\begin{CD}
H_{d-1}(A \cap B) @>h>>  H_{d-1}(A) \oplus H_{d-1}(B) @>g>> H_{d-1}(A \cup B) @>>>\tilde{H}_{d - 2}(A \cap B) 
\end{CD}$
\end{center}

Now $h$ is given by sending the cycle $x$ to $(x, -x)$ in $H_{d-1}(A) \oplus H_{d-1}(B)$ (see, for example, Chapter 4 of \cite{Hatcher}). Furthermore since $A \cap B$ is homeomorphic to a disjoint union of $t$ $(d-1)$-dimensional spheres, $H_{d-1}(A \cap B)$ is free. Thus $h$ is determined by the image of the generators of $H_{d-1}(A \cap B)$. By how we have attached $A$ and $B$ and preserved the initial orientations on both, the image of $h$ is $\langle (\gamma_i, -\tau_i)_{i \in [t]} \rangle$. Now if $d \geq 3$ then $\tilde{H}_{d - 2}(A \cap B)$ is zero since it is the $(d-2)$nd homology group of a disjoint union of $(d-1)$-spheres. Thus for $d \geq 3$, $g$ is surjective, so $H_{d-1}(A \cap B) \cong (H_{d-1}(A) \oplus H_{d-1}(B))/\ker(g) = (H_{d-1}(A) \oplus H_{d-1}(B)) / \Im (h)$, and the lemma follows. On the other hand if $d= 2$ then by connectedness of $A$ and $B$ we have the following short exact sequence:
\begin{center}
$\begin{CD}
0 @>>> (H_1(A) \oplus H_1(B))/\Im(h) @>g'>> H_1(A \cup B) @>>> \tilde{H}_{0}(A \cap B) @>>> 0
\end{CD}$
\end{center}
Here $g'$ is the map induced by $g$ on the quotient group $(H_1(A) \oplus H_1(B)) / \Im(h)$, which is well-defined since $\Im(h) = \ker(g)$. This sequence splits since $\tilde{H}_0(A \cap B)$ is a free abelian group, so $H_1(A \cup B) \cong (H_1(A) \oplus H_1(B)) / \Im(h) \oplus \tilde{H}_0(A \cap B)$. Thus the claim follows in this case too.
 \end{proof}

\subsection{Building the space from $Y_1$ and $Y_2$.}
As discussed in the outline of the proof of Lemma \ref{const}, the ultimate goal of the construction is two simplicial complexes $Y_1$ and $Y_2$ which satisfy certain properties and which may be attached in a way to build a complex which has a prescribed finite cyclic group as the torsion group of a prescribed homology group. Before fully constructing $Y_1$ and $Y_2$ we give sufficient conditions that will allow them to be attached to form $X$ with $H_{d - 1}(X)_T \cong \Z/m\Z$:
\begin{lemma}\label{finalstep}
Fix $d \geq 2$, and let $(n_1, ..., n_k)$ be a list of $k$ nonnegative integers with $n_1 < n_2 < \cdots < n_k$. Suppose that $Y_1$ is a connected, oriented, simplicial complex with $H_{d-1}(Y_1) = \langle \gamma_0, \gamma_1, ..., \gamma_{n_k} \mid 2\gamma_0 = \gamma_1, 2\gamma_1 = \gamma_2, ..., 2\gamma_{n_k - 1} = \gamma_{n_k} \rangle$ where each $\gamma_i$ is coherently represented by a $d$-simplex boundary denoted $Z_i$, with $Z_i \cap Z_j = \emptyset$ for all $i \neq j$. Suppose that $Y_2$ is a connected, oriented, simplicial complex with $H_{d-1}(Y_2) = \langle \tau_1, ..., \tau_k \mid \tau_1 + \tau_2 + \cdots + \tau_k = 0 \rangle$ where each $\tau_i$ is coherently represented by a $d$-simplex boundary denoted $Z'_i$, so that for all $i \neq j$, $Z_i' \cap Z_j' = \emptyset$ and there are no edges between the vertices of $Z_i'$ and the vertices of $Z_j'$. Then $Y_1$ and $Y_2$ may be attached to one another along a subcomplex so that the resulting space is a simplicial complex $X$ so that $H_{d-1}(X)_T$ is isomorphic to the cyclic group of order $2^{n_1} + 2^{n_2} + \cdots + 2^{n_k}$.
\end{lemma}
\begin{proof}
We will apply Lemma \ref{attachinglemma}. For each $i$, let $v_0^i, ..., v_d^i$ denote the coherent ordering of $Z_i$, and let $w_0^i, .., w_d^i$ denote the coherent-ordering of $Z_i'$. Define a simplicial map $f: \bigsqcup_{i = 1}^k Z_i' \rightarrow \bigsqcup_{i = 1}^k Z_{n_i}$ given by $w_j^i \mapsto v_j^{n_i}$ for all $j \in \{0, ..., d\}$ and all $i \in \{1, ..., k\}$. Let $X = Y_1 \sqcup_f Y_2$, then by Lemma \ref{attachinglemma}, $X$ is a simplicial complex, and:
\begin{eqnarray*}
H_{d-1}(X)_T &\cong& [(H_{d-1}(Y_1) \oplus H_{d-1}(Y_2)) / \langle (\gamma_{n_i}, \tau_i)_{i \in [k]} \rangle]_T \\
 &\cong& \langle \gamma_0, \gamma_1, ..., \gamma_{n_k}, \tau_1, ..., \tau_k \mid 2\gamma_0 = \gamma_1, 2\gamma_1 = \gamma_2, ..., 2\gamma_{n_k - 1} = \gamma_{n_k}, \\
&&\tau_1 + \tau_2 + \cdots + \tau_k = 0, \gamma_{n_1} = \tau_1, ..., \gamma_{n_k} = \tau_k \rangle _T\\
&\cong& \langle \gamma_0, \gamma_1, ..., \gamma_{n_k} \mid 2\gamma_0 = \gamma_1, ..., 2\gamma_{n_k - 1} = \gamma_{n_k}, \gamma_{n_1} + \gamma_{n_2} + \cdots + \gamma_{n_k} = 0 \rangle_T \\
&\cong& [\Z/(2^{n_1} + 2^{n_2} + \cdots + 2^{n_k})\Z]_T \\
&=& \Z/(2^{n_1} + 2^{n_2} + \cdots + 2^{n_k})\Z
\end{eqnarray*}
In fact, from Lemma \ref{attachinglemma}, we actually have if $d \geq 3$ then $H_{d-1}(X) \cong \Z/(2^{n_1} + 2^{n_2} + \cdots + 2^{n_k})\Z$, and if $d = 2$ then $H_1(X) \cong \Z^{k - 1} \oplus \Z/(2^{n_1} + 2^{n_2} + \cdots + 2^{n_k})\Z$.
 \end{proof}

Now if we have $Y_1$ and $Y_2$ attached to one another satisfying the statement of Lemma \ref{finalstep}, then the clearly the resulting complex $X$ has $\Delta(X) \leq \Delta(Y_1) + \Delta(Y_2)$ and $|V(X)| \leq |V(Y_1)| + |V(Y_2)|$. But of course, this is irrelevant to checking that $X$ has the prescribed torsion in homology. Nevertheless, these properties are both critical to the second step in the proof of the main theorem. Moving forward it will be our goal to build $Y_1$ and $Y_2$ satisfying the assumptions of Lemma \ref{finalstep}, but with bounded degree and few vertices. 

\subsection{The triangulated sphere $Y_2$}\label{Y2details}
In this section we describe how to build the space $Y_2$. The complex $Y_2$ will be a triangulated $d$-dimensional sphere with $k$ top-dimensional faces removed where $k$ is the Hamming weight of $m$. However we want a bound on the degree of $Y_2$ and have the number of vertices be linear in $k$. Towards that goal we prove the following fact about triangulations of $d$-dimensional spheres.

\begin{lemma} \label{sphere}
For every $d, k \in \N$ there exists a triangulation $T$ of $S^d$ so that $\Delta_{0, 1}(T) \leq (d + 1)(d^2 + d + 2)$, $V(T) = (d^2 + d + 1)k + (d + 2)^2$, and T has $k$ $d$-dimensional faces which are vertex-disjoint and nonadjacent. Therefore for every $d  \in \N$ there exists $L$ so that for any $k$, there is a triangulation $T$ of $S^d$ so that $\Delta(T) \leq L$ and $|V(T)| \leq Lk$ which has $k$ $d$-dimensional faces $t_1, ..., t_k$ which are vertex-disjoint and nonadjacent.
\end{lemma}
\begin{proof}
We first define, for each fixed $d$, an infinite sequence $T_0, T_1, ...$ of triangulations of $S^d$ with $\Delta_{0, 1}(T_i) \leq 2(d + 1)$ and $|V(T_i)| = d + 2 + i$ for all $i$. Let $T_0$ be the $(d +1)$-simplex boundary on the vertex set $\{v_0, v_1, ..., v_{d + 1}\}$. We will build the sequence inductively using bistellar 0-move (first defined by \cite{Pachner1,Pachner2}). Recall that a bistellar 0-moves for a $d$-dimensional simplicial complex is the triangulation obtained by deleting a $d$-dimensional face and replacing it by the cone over its boundary. To obtain $T_1$, we perform a bistellar 0-move at the $d$-dimensional face $[v_1, ..., v_{d + 1}]$ and call the new vertex $v_{d + 2}$. In general we obtain $T_{i + 1}$ from $T_i$ by performing a bistellar 0-move at the $d$-dimensional face $[v_{i + 1}, ..., v_{d + i + 1}]$ with new vertex $v_{d + i + 2}$. We note that this bistellar 0-move is always well-defined since it results in $[v_{i + 2}, ..., v_{d + i + 2}]$ existing in $T_{i + 1}$ and so we may continue inductively. If $d = 2$, then Figure \ref{fig:2sphere} shows successive triangulations of the face $[v_1, v_2, v_3]$ in $T_0, ..., T_4$.\\

By construction, we have that $|V(T_{i + 1})| =|V(T_i)| + 1$ and $|f_d(T_{i+1})| = |f_d(T_i)| + d$. Moreover, by the choice of the face to subdivide at each step, we have that for each vertex we only subdivide a face containing it $(d + 1)$ times in the entire sequence. Each of these subdivisions increases the edge-degree of any vertex by at most 1. Since each vertex has edge-degree $d + 1$ when it is added to the complex, we have that $\Delta_{0, 1}(T_i) \leq 2(d + 1)$ for all $i$. Similarly, when a vertex is added it belongs to $d + 1$ top-dimensional faces, and every bistellar 0-move increases the number of top dimensional faces which contain a fixed vertex by $d - 1$, thus $\Delta_{0, d}(T_i) \leq d + 1 + (d - 1)(d + 1) = (d + 1)d$ for all $i$. We will need this in the next step. \\

Now let $k$ and $d$ be given. Take the triangulation $T_k$ of $S^d$ as defined above. Then $T_k$ has $d + 2 + dk \geq k$ faces, but they are not vertex disjoint nor are they nonadjacent. To fix this, we will perform a sequence of $d + 1$ bistellar 0-moves on every $d$-dimensional face of $T_k$ to reach the final triangulation that we want. At each face $[u_0, ..., u_d] \in f_d(T_k)$, take the bistellar 0-move with new vertex $w_0$, then take a bistellar 0-move at $[u_1, ..., u_d, w_0]$ with new vertex $w_1$, and continue in this way to obtain a new face $[w_0, w_1, ..., w_d]$. These new faces form a family of vertex-disjoint and nonadjacent faces of size $d + 2 + dk \geq k$, so let $T$ be this final triangulation. By this subdivision process on $T_k$, we have added $|f_d(T_k)|(d + 1)$ new vertices and have increased the edge-degree of every vertex in $T_k$ by at most $(d + 1)\Delta_{0, d}(T_k)$ plus we have added new vertices which have edge-degree at most $2(d + 1)$. Therefore we have $|V(T)| = |V(T_k)| + |f_d(T_k)|(d + 1) = d + 2 + k + (d + 1)(d + 2 + dk) = (d^2 + d + 1)k + (d + 2)^2$ and $\Delta_{0, 1}(T) \leq 2(d + 1) + (d + 1)\Delta_{0, d}(T_k) \leq 2(d + 1) + (d + 1)^2d = (d + 1)(d^2 + d + 2)$.\\

The second part of the lemma follows immediately since bounded degree of the 1-skeleton of $T$ implies that $\Delta(T)$ is bounded. 
\end{proof}
\begin{center}
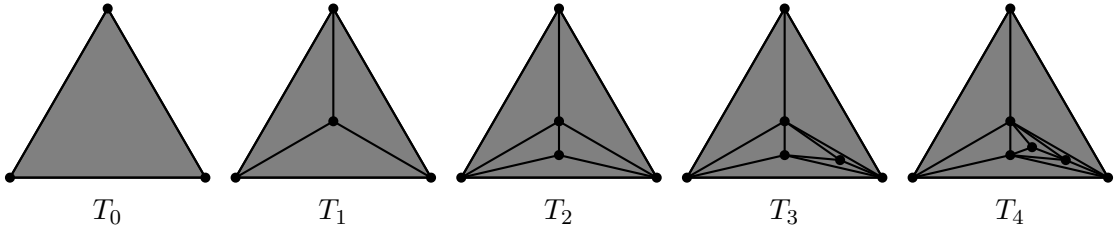
\begin{figure}[h]
\centering
\tikzstyle{every node}=[circle, draw, fill=black,
                        inner sep=0pt, minimum width=3pt]
\begin{centering}                        
\begin{tikzpicture}[thick,scale=3]
\foreach \x in {-2, -1, 0, 1, 2}
{
	\begin{scope}[xshift = \x cm ]
		\draw[fill = gray] (90 : 0.5) -- (210:0.5) -- (330:0.5) -- cycle;
		\draw (90:0.5) -- (210:0.5);
		\draw(210:0.5) -- (330:0.5);
		\draw(330:0.5) -- (90:0.5);
		\draw (90: 0.5) node{};
		\draw (210:0.5) node{};
		\draw (330:0.5) node{};
	\end{scope}
	
}

\foreach \x in {-1, 0, 1, 2}
{
	\begin{scope}[xshift = \x cm ]
		\draw (90:0) -- (90:0.5);
		\draw  (90:0) -- (210:0.5);
		\draw (90:0) -- (330:0.5);
		\draw (90: 0) node{};
		
	\end{scope}
	
}

\foreach \x in {0, 1, 2}
{
	\begin{scope}[xshift = \x cm ]
		\draw (90:-0.15) -- (90:0);
		\draw  (90:-0.15) -- (210:0.5);
		\draw (90:-0.15) -- (330:0.5);
		\draw (90: -0.15) node{};
		
	\end{scope}
	
}

\foreach \x in {1, 2}
{
	\begin{scope}[xshift = \x cm ]
		\draw (145:-0.3) -- (90:0);
		\draw  (145:-0.3) -- (90:-0.15);
		\draw (145:-0.3) -- (330:0.5);
		\draw (145: -0.3) node{};
		
	\end{scope}
	
}

\foreach \x in {2}
{
	\begin{scope}[xshift = \x cm ]
		\draw (130:-0.15) -- (90:0);
		\draw  (130:-0.15) -- (90:-0.15);
		\draw (130:-0.15) -- (145:-0.3);
		\draw (130: -0.15) node{};
		
	\end{scope}
	
}
\tikzstyle{every node} = []
\draw (-2, -0.4) node{$T_0$};
\draw (-1, -0.4) node{$T_1$};
\draw (0, -0.4) node{$T_2$};
\draw (1, -0.4) node{$T_3$};
\draw (2, -0.4) node{$T_4$};
\end{tikzpicture} \\
\end{centering}
\caption{The first few steps of the triangulation described in the proof of Lemma \ref{sphere} when $d = 2$}\label{fig:2sphere}
\end{figure}
\end{center}
Next we construct $Y_2$ from the claim above.
\begin{lemma}\label{Y2lemma}
Fix a dimension $d \geq 2$ and let $L$ be the constant associated to $d$ from Lemma \ref{sphere}. Given an integer $k$ there exists a connected, oriented simplicial complex denoted $Y_2$ so that $H_{d-1}(Y_2)$ is presented as the abelian group $\langle \tau_1, ..., \tau_k \mid \tau_1 + \tau_2 + \cdots + \tau_k = 0 \rangle$ where each $\tau_i$ is coherently represented by a $d$-simplex boundary, which are vertex-disjoint and nonadjacent to one another, and so that $\Delta(Y_2) \leq L$ and $|V(Y_2)| \leq 2Lk$.
\end{lemma}
\begin{proof}
Let $L$ be the constant depending on $d$ found in Lemma \ref{sphere} and let $T$ be a triangulation of $S^{d}$ with $\Delta(T) \leq L$ and $|V(T)| \leq L(2k)$ which has $2k$ vertex-disjoint and nonadjacent $d$-dimensional faces, which we will denote by $t_1, t_2, ..., t_{2k}$. Now give $T$ an orientation according to some ordering on the vertices. Since $T$ is a triangulation of the $d$-dimensional sphere there exists a $d$-chain $x = (x_1, ...., x_l)$ (where $l = |f_d(T)|$) so that $\partial_d x = 0$. Furthermore since every codimension-1 face of $T$ has degree 2 (since $T$ is a triangulated manifold), we may assume that $x$ is a $d$-chain with all coefficients equaling $-1$ or 1. Thus without loss of generality the coefficient on at least $k$ of the $2k$ faces $t_1, t_2, ..., t_{2k}$ is $-1$. Then if we delete $k$ of these faces we get the oriented simplicial complex that we want, $Y_2$ which has homology $H_{d - 1}(Y_2)$ presented as $\langle \tau_1, ..., \tau_k \mid \tau_1 + \tau_2 + \cdots + \tau_k = 0 \rangle$ where each $\tau_i$ is coherently represented by the $d$-simplex boundary of a unique deleted face. Clearly $\Delta(Y_2) \leq \Delta(T) \leq L$, $|V(Y_2)| = |V(T)| \leq 2Lk$, and $Y_2$ is connected.
 \end{proof}

\subsection{The triangulated telescope $Y_1$}\label{Y1details}
In this section we will describe how to build, for any dimension $d \geq 2$ and any integer $n$, a $d$-dimensional, connected, oriented simplicial complex $Y_1$ so that $H_{d - 1}(Y_1)$ is presented by $\langle \gamma_0, \gamma_1, \gamma_2, ... \gamma_n \mid 2 \gamma_0 = \gamma_1, 2\gamma_1 = \gamma_2, \cdots, 2 \gamma_{n - 1} =\gamma_n \rangle$ where each $\gamma_i$ is coherently represented by a $d$-simplex boundary, which are vertex disjoint from one another, and so that $\Delta(Y_1) \leq M$ and $|V(Y_1)| \leq Mn$ for some constant $M$ depending only on $d$. Note that $n$ will be $n_k$ from the binary expansion of $m$, that is $n$ will be the number of bits in the binary expansion of $m$.

The construction of $Y_1$ will be accomplished by first constructing a simplicial complex $P = P(d)$ with an orientation so that $H_{d - 1}(P) = \langle a, b \mid 2a = b \rangle$ where $a$ and $b$ are coherently represented by a pair of vertex-disjoint $d$-simplex boundaries. We may then attach $n$ copies of these complexes ``end-to-end" to build $Y_1$, with $\Delta(Y_1) \leq 2\Delta(P)$ and $|V(Y_1)| \leq |V(P)|n$. We now build $P(d)$.\\

\begin{lemma}\label{Pdlemma}
Fix $d \geq 2$. Then there exists an oriented simplicial complex $P$ depending only on $d$, with its orientation induced by an ordering on the vertices, so that $H_{d - 1}(P)$ is presented as $\langle a, b \mid 2a = b \rangle$ where $a$ and $b$ are coherently represented by a pair of vertex-disjoint $d$-simplex boundaries.
\end{lemma}
\begin{proof}
The proof will be by induction on $d$. For $d = 2$ our complex will be the pure simplicial complex on vertex set $\{1, 2, 3, 4, 5, 6\}$ with orientation induced by the natural ordering on the vertices and top dimensional faces $[1, 2, 6]$, $[1, 3, 6]$, $[3, 5, 6]$, $[2, 4, 6]$, $[2, 3, 4]$, $[1, 3, 4]$, $[1, 4, 5]$, $[1, 2, 5]$, and $[2, 3, 5]$. This complex is given as Figure 4. Of course it matches Figure 2, but the vertices have been relabeled as the focus here is only to describe the building block, but not how they are attached to one another.
\begin{center}
\begin{figure}[h]\label{fig:P2}
\centering
\tikzstyle{every node}=[circle, draw, fill=black,
                        inner sep=0pt, minimum width=4pt]
\begin{centering}                        
\begin{tikzpicture}[thick,scale=1.75]
\draw[fill = gray] (90: 1) -- (150: 1) -- (210:1) -- (270:1) -- (330:1) -- (390: 1) -- cycle;
\draw[fill = white] (90:0.5) -- (210:0.5) -- (330:0.5) -- cycle;
\foreach \x in {90, 150, ..., 390}
{
\draw (\x:1)--(\x+60:1);
}

\foreach \x in {90, 210, 330}
{
	\draw (\x:0.5) -- (\x + 120:0.5);
	\draw (\x: 0.5) -- (\x: 1);
}

\foreach \x in {90, 210, 330}
{
	\draw(\x:0.5) -- (\x + 60: 1);
	\draw(\x:0.5) -- (\x - 60: 1);
}

\foreach \x in {90, 210, 330}
{
	\draw (\x:0.5) node{};
}
\foreach \x in {90, 150, ..., 390}
{
\draw(\x:1) node {};
}
\tikzstyle{every node}=[]

{
\tiny
\draw (90:1.15) node {$1$};
\draw(150:1.15) node{$2$};
\draw(210:1.15) node{$3$};
\draw (270:1.15) node {$1$};
\draw(330:1.15) node{$2$};
\draw(390:1.15) node{$3$};
}

{
\tiny
\draw (90:0.35) node {$6$};
\draw(210:0.35) node{$4$};
\draw(330:0.35) node{$5$};

}


\end{tikzpicture} 
\end{centering}
\caption{The simplicial complex $P(2)$}
\end{figure}
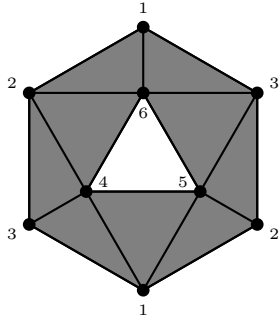
\end{center}

Observe that $\partial_2([1, 2, 6] + [2, 4, 6] + [2, 3, 4] - [1, 3, 4] - [1, 4, 5] + [1, 2, 5] + [2, 3, 5] - [3, 5, 6] - [1, 3, 6]) = 2([1, 2] - [1, 3] + [2, 3]) - ([4, 5] - [4, 6] + [5, 6])$. Thus $H_1(P)$ can be presented as $\langle a, b \mid 2a = b\rangle$ where $a$ is represented by $[1, 2] - [1, 3] + [2, 3]$ and $b$ is represented by $[4, 5] - [4, 6] + [5, 6]$. That is, $a$ is coherently represented by the triangle boundary on vertex set $\{1, 2, 3\}$ and $b$ is coherently represented by the triangle boundary on vertex set $\{4, 5, 6\}$. This completes the base case.\\

We are now ready to prove the inductive step; to build $P(d + 1)$ from $P(d)$. Begin with $P(d)$ a $d$-dimensional oriented simplicial complex, with the orientation induced by an ordering on the vertices, with $H_{d - 1}(P(d))$ presented as $\langle a, b\mid 2a = b \rangle$ with $a$ and $b$ coherently represented by vertex-disjoint $d$-simplex boundaries. Write $a_0, a_1 ..., a_d$ for the coherently-ordered vertices of the simplex boundary representing $a$ and $b_0, b_1, ..., b_d$ for the coherently-ordered vertices of the simplex boundary representing $b$. Now take the suspension of $P(d)$ with two suspension vertices $w_1$ and $w_2$. When taking the suspension $SP(d)$ keep the ordering on the vertices of $P(d)$ and add $w_1$ followed by $w_2$ to the beginning of the ordering.\\

 With respect to this ordering we have that the map $\phi: C_i(P(d)) \rightarrow C_{i + 1}(SP(d))$ given by sending the generator $[v_0, ..., v_i]$ to $[w_1, v_0, ..., v_i] - [w_2, v_0, ..., v_i]$ induces an isomorphism from $H_{i}(P(d))$ to $H_{i+1}(SP(d))$ for $i > 0$, in particular for $i = d - 1$ (This is routine to check, but we do prove it as the claim below.). Thus, by the inductive hypothesis $H_d(SP(d))$ is generated by $a', b'$ with the relator $2a' = b'$ and where $a'$ is represented by 
$$\sum_{i = 0}^d (-1)^i [w_1, a_0, ..., \hat{a_i}, ..., a_d] - \sum_{i = 0}^{d} (-1)^i [w_2, a_0, ..., \hat{a_i}, ..., a_d],$$ and $b'$ is represented by 
$$\sum_{i = 0}^d (-1)^i [w_1, b_0, ..., \hat{b_i}, ..., b_d] - \sum_{i = 0}^{d} (-1)^i [w_2, b_0, ..., \hat{b_i}, ..., b_d].$$

While we have that $H_d(SP(d)) = \langle a', b' \mid 2a' = b' \rangle$, we are not done since $a'$ and $b'$ are not represented by $(d + 1)$-simplex boundary. To fix this we add the $(d + 1)$-dimensional faces $[w_1, a_0, ..., a_d]$ and $[w_2, b_0, ..., b_d]$ along with the necessary $d$-dimensional faces $[a_0, ..., a_d]$ and $[b_0, ..., b_d]$ to our complex. Now we have that 
$$\partial_{d + 1}[w_1, a_0, ..., a_d] = [a_0, ..., a_d] + \sum_{i = 0}^d (-1)^{i + 1} [w_1, a_0, ..., \hat{a_i}, ..., a_d],$$
and 
$$\partial_{d + 1}[w_2, b_0, ..., b_d] = [b_0, ..., b_d] + \sum_{i = 0}^d (-1)^{i + 1} [w_2, b_0, ..., \hat{b_i}, ..., b_d].$$
Thus after adding in these two new $(d + 1)$-dimensional faces and two $d$-dimensional faces we have the new relators in the codimension-1 homology group of our complex given by 
$$[a_0, ..., a_d] = \sum_{i = 0}^d (-1)^{i} [w_1, a_0, ..., \hat{a_i}, ..., a_d]$$
and
$$[b_0, ..., b_d] = \sum_{i = 0}^d (-1)^{i} [w_2, b_0, ..., \hat{b_i}, ..., b_d]$$
Using these relators and the representatives for $a'$ and $b'$ listed above, we have that the codimension-1 homology group is generated by 
$$a' =  [a_0, ..., a_d] + \sum_{i = 0}^{d} (-1)^{i + 1} [w_2, a_0, ..., \hat{a_i}, ..., a_d]$$
and 
$$b' = -[b_0, ..., b_d] - \sum_{i = 0}^d (-1)^{i + 1}[w_1, b_0, ..., \hat{b_i}, ..., b_d]$$
with the relator $2a' = b'$. Moreover we have that $a'$ and $-b'$ are coherently represented by the boundary of $(d + 1)$-simplicies, namely the boundaries of $[w_2, a_0, ..., a_d]$ and $[w_1, b_0, ..., b_d]$ respectively. We are not quite finished yet since the orientation has been reversed, but if we simply reverse the order of $b_0$ and $b_1$ then we have the oriented simplicial complex $P(d + 1)$ that we want. Alternatively, we could observe that there is a sign change every time we increase the dimension and modify $P(2)$ by switching the labels 4 and 5 if we are building to an odd dimension $d$. 

 \end{proof}

In the previous proof we make use of the following claim, which we prove here for the sake of completeness.
\begin{unnumclaim}\label{isomorphism}
Suppose $X$ is an oriented simplicial complex with orientation induced by an ordering on the vertices of $X$. If we take the suspension of $X$, denoted $SX$, with the two suspension vertices $w_1$ and $w_2$ added to the beginning of the vertex ordering, then the map $\phi: C_i(X) \rightarrow C_{i + 1}(SX)$ given by sending each generator $[v_0, ..., v_i]$ to $[w_1, v_0, ..., v_i] - [w_2, v_0, ..., v_i]$ induces an isomorphism from $H_i(X)$ to $H_{i + 1}(X)$ for all $i > 0$.
\end{unnumclaim}
\begin{proof}
Let $\partial_i$ denote the $i$th boundary map of $X$ and $\partial_i'$ denote the $i$th boundary map of $X'$. By the choice of ordering we have for each $i > 0$ that the matrix $\partial_{i + 1}'$ is given by 
\begin{center}
$\partial_{i + 1}' = 
\begin{pmatrix}
- \partial_i & 0 & 0  \\
0 & -\partial_i & 0 \\
I & I  & \partial_{i + 1}\\
\end{pmatrix},$
\end{center}
where the columns are indexed by $(i + 1)$-dimensional faces which contain $w_1$, followed by $(i + 1)$-dimensional faces that contain $w_2$, followed by $(i + 1)$-dimensional faces present in $X$, and the rows are indexed by $i$-dimensional faces that contain $w_1$, followed by $i$-dimensional faces that contain $w_2$, followed by $i$-dimensional faces present in $X$. With respect to this basis, the map $\phi$ sends an arbitrary vector $v = (a_0, ..., a_k) \in C_i(X)$ to the vector $(a_0, ..., a_k, -a_0, ..., -a_k, 0, ... 0)$ in $C_{i + 1}(SX)$, where $k$ is the number of $i$-dimensional faces in $X$. To simplify notation we denote $\phi(v)$ with $(v, -v, 0)$.\\

Now we prove that $\phi$ is well-defined on homology groups by showing that $\phi$ sends cycles to cycles and boundaries to boundaries. Suppose that $\partial_i(v) = 0$, then by the construction of $\partial_{i + 1}'$ given above we have that $\partial_{i + 1}'(\phi(v)) = \partial_{i + 1}'((v, -v, 0)) = 0$, so $\phi$ sends cycles to cycles. Next, suppose that $v$ is in the image of $\partial_i$, then there exists $u$ so that $\partial_i u = v$. Thus, $\partial_{i + 1}'((-u, u, 0)) = (v, -v, 0) = \phi(v)$. It follows that $\phi$ is a well-defined homomorphism on homology groups.\\

Now we check that $\phi$ is injective. Suppose $\phi(v) = 0$. Then $(v, -v, 0)$ belongs to $\Im(\partial_{i + 2}')$. Thus there exists $u$ which we write as $(u_1, u_2, u_3)$ so that $\partial_{i + 2}'((u_1, u_2, u_3)) = (v, -v, 0)$. Thus $\partial_{i + 1}(-u_1) = v$, so $v \in \Im(\partial_{i + 1})$. Thus at the level of homology groups $v = 0$, so $\phi$ is injective.\\

Now we show that $\phi$ is surjective. Let $z \in H_{i + 1}(SX)$. Since every $(i+1)$-dimensional face of $X$ is contained in at least one (in fact at least 2) $(i + 2)$-dimensional faces in $SX$, we may write $z$ as $x + y$ where $x$ is a sum of generating $(i + 1)$-chains that all contain the vertex $w_1$ and $y$ is a sum of generating $(i + 1)$-chains which all contain the vertex $w_2$. That is we may write $z = (x, y, 0)$. Since $\partial_{i + 1}'(z) = 0$ we have $\partial_{i + 1}'((x, y, 0)) = (-\partial_i x , -\partial_i y, x + y) = (0, 0, 0)$. Thus $x = -y$ in the sense that after deleting $w_1$ from every chain in the support of $x$ and deleting $w_2$ from every chain in the support of $y$ we have that $x$ and $-y$ are the exact same $i$-chain. Thus $\phi(x) = (x, -x, 0) = (x, y, 0) = z$, proving that $\phi$ in onto, and completing the proof that $\phi$ is an isomorphism. 
\end{proof}
Now that we have the construction for $P(d)$ we may apply Lemma \ref{attachinglemma} to construct the complex $Y_1$ that we need.
\begin{lemma}\label{Y1lemma}
Fix $d \geq 2$, and let $P = P(d)$ denote the complex constructed in Lemma \ref{Pdlemma}, and let $n$ be a positive integer. Then there exists a connected, oriented simplicial complex $Y_1$ with $\Delta(Y_1) \leq 2 \Delta(P)$, $|V(Y_1)| \leq n |V(P)|$, and $H_{d - 1}(Y_1)$ presented by $\langle \gamma_0, ..., \gamma_n \mid 2 \gamma_0 = \gamma_1, 2 \gamma_1 = \gamma_2, ..., 2 \gamma_{n - 1} = \gamma_n \rangle$ where each $\gamma_i$ is coherently represented by a $d$-simplex boundary $Z_i$ so that the for all $i \neq j$, $Z_i \cap Z_j = \emptyset$.
\end{lemma}
\begin{proof}
Fix $n$ and $d$. Take $n$ copies of $P$ denoted $P_1, P_2, ..., P_{n}$ with $H_{d - 1}(P_i) = \langle a_i, b_i \mid 2 a_i = b_i\rangle$, $a_i$ coherently represented by $d$-simplex boundary $A_i$ and $b_i$ coherently represented by $d$-simplex boundary $B_i$. Now use Lemma \ref{attachinglemma} to attach $P_i$ to $P_{i + 1}$ by the order preserving simplicial homeomorphism $f_i : B_i \rightarrow A_{i + 1}$ for every $i \in \{1, ..., n - 1\}$. This results in a connected, oriented simplicial complex $Y_1$ which has $H_{d - 1}(Y_1)$ presented by $\langle \gamma_0, ..., \gamma_n \mid 2 \gamma_0 = \gamma_1, 2 \gamma_1 = \gamma_2, ..., 2 \gamma_{n - 1} = \gamma_n \rangle$ so that each $\gamma_i$ is coherently represented by a $d$-simplex boundary which are all vertex disjoint from one another. Furthermore no vertex belongs to more than two copies of $P$, and so $\Delta(Y_1) \leq 2 \Delta(P)$. Moreover, it is clear that $|V(Y_1)| \leq n|V(P)|$ since $Y_1$ is built out of $n$ copies of $P$.
 \end{proof}

\subsection{Finishing the proof}
\begin{proof}[Proof of Lemma \ref{const}]
Fix the dimension $d \geq 2$. We prove that the constant $K = \max\{2 \Delta(P) + L + 1, 2|V(P)| + 4L \}$ satisfies the conclusion where $L$ is the constant depending only on $d$ from Lemma \ref{sphere} and $P$ is the simplicial complex depending only on $d$ from Lemma \ref{Pdlemma}. Let $m \geq 2$ be given. Write $m$ in its binary expansion $m = 2^{n_1} + 2^{n_2} + \cdots + 2^{n_k}$ where $0 \leq n_1 < n_2 < \cdots < n_k$. Note that $n_k \leq \log_2 m$ and $k \leq \log_2 m + 1$. By Lemma \ref{Y1lemma} with $n = n_k$ there exists a connected, oriented simplicial complex $Y_1$ with $\Delta(Y_1) \leq 2 \Delta(P)$, $|V(Y_1)| \leq n_k |V(P)|$, and $H_{d - 1}(Y_1)$ presented by $\langle \gamma_0, \gamma_1, \cdots, \gamma_{n_k} \mid 2 \gamma_0 = \gamma_1, 2 \gamma_1 = \gamma_2, \cdots 2 \gamma_{n_k - 1} = \gamma_{n_k} \rangle$ where each $\gamma_i$ is coherently represented by a $d$-simplex boundary all of which are vertex-disjoint from one another. \\

Next by Lemma \ref{Y2lemma} there exists a connected, oriented simplical complex $Y_2$ with $\Delta(Y_2) \leq L$ and $|V(Y_2)| \leq 2Lk$ so that $H_{d - 1}(Y_2)$ is presented as $\langle \tau_1, \tau_2, ..., \tau_k \mid \tau_1 + \tau_2 + \cdots + \tau_k = 0 \rangle$ where each $\tau_i$ is coherently represented by a $d$-simplex boundary, which are vertex disjoint and nonadjacent to one another. \\

Now $Y_1$ and $Y_2$ satisfy the assumptions of Lemma \ref{finalstep} with the list of nonnegative integers $(n_1, n_2, ..., n_k)$, so they may be glued together along a subcomplex to build a simplicial complex $X$ with $H_{d - 1}(X)_T \cong \Z / (2^{n_1} + 2^{n_2} + \cdots + 2^{n_k}) \Z = \Z / m\Z$. Finally, since $X$ is build by attaching $Y_1$ to $Y_2$ along a subcomplex we have $$\Delta(X) \leq \Delta(Y_1) + \Delta(Y_2) \leq 2\Delta(P) + L \leq K - 1$$ and $$|V(X)| \leq |V(Y_1)| + |V(Y_2)| \leq |V(P)|n_k + 2Lk \leq (2|V(P)| + 4L)\log_2 m \leq K \log_2 m.$$
 
\end{proof}

\section{A remark about cohomology}
Given a dimension $d \geq 2$ and a finite abelian group $G$, the proof of Lemma \ref{general} builds a simplicial complex $X$ whose homology groups may all be computed. If $G \cong \Z/m_1 \Z \oplus \Z/ m_2 \Z \oplus \cdots \oplus m_l \Z$, with $m_1 | m_2 | \cdots | m_l$ then $H_i(X) = 0$ for all $i \notin \{0, 1, d - 1\}$, $H_0(X) \cong \Z^l$, $H_1(X) \cong \Z^{k_1 + k_2 + \cdots + k_l - l}$ where $k_i$ denotes the Hamming weight of $m_i$, and $H_{d - 1}(X) \cong G$ (take the direct sum of these last two if $d = 2$). This can all be checked routinely using the Mayer--Vietoris sequence, but as we are only interested in $H_{d - 1}(X)_T$ for the main result, we do not include the details about the computation of the other homology groups. \\

Even though we know all the homology groups of $X$, after coloring the vertices of $X$ by $c$ as in Lemma \ref{colorings} using the probabilistic method, we have no control over $H_i((X, c))$ for $i < d - 1$ nor the free part of $H_{d - 1}((X, c))$. However, we observe that under the assumptions of Lemma \ref{reduce}, $H_d((X, c)) \cong H_d(X)$. Thus by the universal coefficient theorem we have that $H^d((X, c)) \cong H^d(X)$. Since $H_d(X) = 0$, our proof gives a stronger version of the main theorem in terms of cohomology as follows:
\begin{unnumtheorem}[Cohomological statement of Theorem 1]
Let $d \geq 2$, then there exists constants $c_d$ and $C_d$ so that for any finite abelian group $G$,
$$c_d(\log |G|)^{1/d} \leq T^d(G) \leq C_d(\log |G|)^{1/d},$$
where $T^d(G)$ denotes the minimum number of vertices $n$ so that there exists a simplicial complex $X$ on $n$ vertices with $H^d(X)$ isomorphic to $G$.
\end{unnumtheorem}

\section{Concluding remarks and open problems}
While Theorem 1 is best possible in terms of the growth of $T_d(G)$ with $d$ fixed, the constants which appear in the proof are by no means optimal. Indeed, as stated in the proof of Theorem 1 the constant $C_d$ is given by 
$$C_d = \frac{18K^{8 + d^{-1}}d^6}{\sqrt[d]{\log{2}}},$$
where $K$ is assumed to be at least 5 and $d$ is assumed to be at least 2. However the initial construction in Lemma 1 gives a construction with at most $K \log_2(|G|)$ where $K$ is the same $K$ which appears in the calculation of $C_d$. Thus we have that 
$$C_d \sqrt[d]{\log|G|} \leq K \log_2|G|$$
only if 
$$\log_2(|G|) \geq \left(18K^{7 + d^{-1}}d^6\right)^{d/(d - 1)}$$
By the assumption on $K$ and $d$, we have that the size of $G$ has to be at least $2^{90{,}000{,}000} $ in order for the bound given by the final construction to even possibly be better than the bound given by the initial construction.  So it is computationally infeasible to compute meaningful upper bound on $T_d(G)$ for any group $G$ of reasonable size purely from the statement of Theorem 1. \\

On the other hand, we could go through the proof and be more careful with our bounds on $\Delta(Y_1)$, $\Delta(Y_2)$, $|V(Y_1)|$, and $|V(Y_2)|$ in an attempt to improve the constant. However, we still have a 3-step approach to finding our coloring $c$, and computations suggest that for small groups $G$ this approach is fairly wasteful.\\

Moreover, the proof of Theorem 1 is non-constructive and so it also doesn't provide a way to actually give a construction that provides the upper bound on $T_d(G)$. Now there are algorithmic versions of the Lov\'asz Local Lemma (see for example \cite{Moser}) which can adapt the use of the Lov\'asz Local Lemma here to give a fully constructive proof, however this still leaves the problem that the group must be extremely large as above in order to be able to directly apply Theorem 1.\\

 Nevertheless the strategy of the proof does give a way to actually construct, for any group $G$, a complex which provides an upper bound on $T_d(G)$. One simply builds the initial construction and then properly colors the vertices according to the rule that no two $(d - 1)$-dimensional faces receive the same pattern. Thus one could apply a type of greedy-coloring approach in order to give explicit constructions of complexes which bound $T_d(G)$ for any group $G$. While this greedy approach may not be asymptotically best possible, early experimental evidence suggests that it works reasonably well, at least when $G$ is a cyclic group. Table \ref{tbl:examples} gives the results of this strategy to upper bound $T_d(\Z/m\Z)$. For each $m$ and $d$ we have a column which provides the number of vertices in the initial construction $X$ (essentially as in the proof of Lemma \ref{const}, though a bit more efficient in triangulating $Y_2$) and a second column that gives the vertices in the final construction $(X, c)$ where $c$ is a coloring found by a greedy algorithm which gives a proper coloring of $V(X)$ so that no two $(d - 1)$-dimensional faces receive the same pattern. Thus $V((X, c))$ is an upper bound on $T_d(\Z/m\Z)$. The natural question is the following:
\begin{question}
Can the bound on $C_d$ be improved by finding a feasible algorithm to properly color the vertices in the initial construction so that no two $(d-1)$-dimensional faces receive the same pattern?
\end{question}
\begin{table}[h]
  \centering
  \renewcommand{\arraystretch}{1.2}
  \begin{tabular}{|c|c|c|c|c|c|c|c|c|}
    \hline
    \multirow{2}{1cm}{\hfill$m$\hfill} & \multicolumn{2}{c|}{$d = 2$} & \multicolumn{2}{c|}{$d = 3$} & \multicolumn{2}{c|}{$d = 4$} & \multicolumn{2}{c|}{$d = 5$} \\ \cline{2-9}
	 & $|V(X)|$ & $|V((X, c))|$ & $|V(X)|$ & $|V((X, c))|$ & $|V(X)|$ & $|V((X, c))|$ & $|V(X)|$ & $|V((X, c))|$ \\ \hline
$10^{10}$ & 115 & 46 & 147 & 39 & 180 & 43 & 214 & 61\\
$10^{25}$ & 273 & 69 & 352 & 51 & 434 & 52 & 518 & 76\\
$10^{50}$ & 561 & 101 & 710 & 66 & 869 & 62 & 1031 & 82\\
$10^{100}$ & 1106 & 142 & 1406 & 82 & 1722 & 70 & 2046 & 86\\
$10^{250}$ & 2789 & 223 & 3524 & 109 & 4307 & 86 & 5110 & 91\\
$10^{500}$ & 5576 & 307 & 7042 & 134 & 8605 & 99 & 10208 & 95\\
$10^{1000}$ & 11131 & 432 & 14067 & 168 & 17196 & 118 & 20403 & 103 \\
$10^{2018}$ & 22461 & 609 & 28384 & 209 & 34698 & 138 & 41169 & 112 \\
     \hline

  \end{tabular}
  \caption{Greedy Approach to Bound $T_d(\Z/m\Z)$. For each $d$, $X$ denotes the initial construction and $(X, c)$ denotes the final construction.}\label{tbl:examples}
\end{table}

A second way one could extend the result here is to prove a $\Q$-acyclic version of Theorem 1. That is given $G$ and $d$, provide a construction of a $d$-dimensional $\Q$-acyclic complex $X$ with $H_{d - 1}(X) \cong G$. In the proof of Theorem 1, we do end up constructing a complex which has trivial top homology group. Moreover, we can simply add in all the missing $(d - 1)$-dimensional faces and not change the torsion part of $H_{d - 1}(X)$. The key obstacle to proving that we may give a construction of a $\Q$-acyclic complex with prescribed torsion group is that after building the construction in the proof of Theorem 1 and filling in all the $(d - 1)$-dimensional faces, we will likely have $\beta_{d - 1} > 0$. Now as long as $\beta_{d - 1} > 0$ one could find a $d$-cell that may be filled in to drop $\beta_{d - 1}$, but we have no guarantee that doing so will not increase the size of the torsion group. However, the initial torsion group will at least be contained in the final torsion group in this case which gives use the following partial result toward a $\Q$-acyclic version as a corollary to our proof of Theorem 1.
\begin{corollary}
For every $d \geq 2$, there exists a constant $C_d$ so that for any finite abelian group $G$, there is a $d$-dimensional $\Q$-acyclic complex $X$ on at most $C_d \sqrt[d]{\log|G|}$ vertices with $G \leq H_{d - 1}(X)$.
\end{corollary}
In particular, this theorem implies that for each $d$ there is a constant $C_d$ so that for every $n$ one has that $p$-torsion is possible in a $d$-dimensional $\Q$-acyclic complex on $n$ vertices for \emph{every} prime $p$ of size at most $\exp(n^d/C_d)$. The question for a $\Q$-acyclic version of Theorem 1 is:
\begin{question}
For every $d \geq 2$ is there a constant $C_d$ so that for every group $G$ there exists a $\Q$-acyclic complex $X$ with $H_{d - 1}(X) \cong G$ and $|V(X)| \leq C_d\sqrt[d]{\log |G|}$?
\end{question}

A final open problem related to Theorem 1 is the following:
\begin{question}\label{FixNFixD}
For a fixed $n$ and $d$ what is the largest abelian group $G$ which can be realized as the torsion part of the $(d - 1)$st homology group for a $d$-dimensional simplicial complex on $n$ vertices?
\end{question}

It is clear that the answer to Question \ref{FixNFixD} is realized by a $d$-dimensional $\Q$-acyclic complex on $n$ vertices. Indeed adding the restriction that the $(d - 1)$-skeleton of the complex with maximum torsion group size is complete does not affect the problem, and deleting a face which drops $\beta_d$ cannot decrease the size of the torsion group, nor can adding a face which drops $\beta_{d - 1}$. However, even for small $n$ and $d$, this problem appears to be highly nontrivial. For example, the group $\Z/3\Z \times (\Z/30\Z)^5$ of size 72,900,000 is realizable as $H_3(X)$ for $X$ a $4$-dimensional complex on only 11 vertices. This example came from a sum complex, as defined in \cite{LMR}, in this case the sum complex $X_{\{1, 2, 3, 5, 8\}}$ on 11 vertices.  Moreover, even random examples as in the example on 16 vertices above or those in Table \ref{tbl:torsionburst} give enormous torsion groups. In fact, the 5-dimensional example given above isn't even close to the largest possible torsion group on 16 vertices. By increasing the dimension to 7 one can find, from the torsion burst in the Linial--Meshulam model, a complex on $16$ vertices with torsion part of the sixth homology group equal to a cyclic group of size about $4.6096 \times 10^{286}$.



\bibliography{ResearchBibliography}
\bibliographystyle{amsplain}

\end{document}